\documentclass[11pt]{article}

\usepackage[a4paper,margin=2.6cm]{geometry}
\usepackage{amsmath,amssymb,amsthm,mathtools}
\usepackage{booktabs}
\usepackage{enumitem}
\usepackage{microtype}
\usepackage[hidelinks]{hyperref}
\usepackage{bm}
\usepackage{multirow}

\newtheorem{theorem}{Theorem}[section]
\newtheorem{lemma}[theorem]{Lemma}
\newtheorem{corollary}[theorem]{Corollary}

\theoremstyle{remark}
\newtheorem{remark}[theorem]{Remark}

\title{Optimal Energy-Norm
Convergence for the Dynamic Diffusion Finite Element Method}

\author{
Isaac P. Santos\thanks{Corresponding author:
\href{mailto:isaac.santos@ufes.br}{\texttt{isaac.santos@ufes.br}}}
\\ 
\small Applied Mathematics Department (DMA)\\
\small Graduate Program in Informatics (PPGI)\\
\small Federal University of Espírito Santo (UFES), Brazil
}

\date{September 2026}

\begin{document}
\maketitle

\begin{abstract}
We revisit the nonlinear two-scale Dynamic Diffusion (DD) finite element
formulation mathematically analyzed by Santos et al. (2021) for stationary
advection--diffusion--reaction problems. We establish an explicit local
Lipschitz estimate for the artificial diffusivity, with a constant of order
$h_T$, and use it to prove uniqueness of the discrete solution for
sufficiently fine meshes. By separating the approximation error in the
energy norm from the contribution associated with artificial diffusion, we
derive an optimal first-order a priori energy-norm estimate for continuous
piecewise linear finite elements enriched with simplex bubble functions. We
further prove that the square root of the nonlinear artificial dissipation is
$O(h)$. Consequently, the combined
energy-error and artificial-dissipation measure also converges with first
order, sharpening the previously available $O(h^{1/2})$ estimate. A smooth
manufactured problem, considered with and without reaction, corroborates the
predicted convergence rates. A second, strongly advection-dominated problem
with sharp outflow layers illustrates the stabilizing behavior of the DD
formulation.
\end{abstract}


\section{Introduction}

The numerical approximation of advection-dominated transport problems
remains challenging due to the presence of sharp internal and boundary
layers and the resulting multiscale character of the solution. In such
regimes, standard Galerkin finite element discretizations may exhibit
nonphysical oscillations, motivating the development of a wide variety
of stabilized formulations, including streamline-upwind, residual-based,
multiscale, and nonlinear shock-capturing methods
\cite{brooks-hughes:82,hughes-etal:04,galeao-dudu:88,knopp-etal:02,john-knobloch:07,john-knobloch:08,Barrenechea-etal:2017}. These approaches have
provided effective stabilization strategies for advection-dominated
problems. Nevertheless, the construction of methods that introduce
sufficient dissipation to control spurious oscillations while avoiding
excessive artificial diffusion remains an important issue, particularly
in the presence of sharp layers. In nonlinear stabilization strategies,
this balance is closely related to the mechanism used to determine the
amount and location of the added diffusion.

The Dynamic Diffusion (DD) methodology emerged within this context as a
nonlinear two-scale finite element approach in which the amount of
additional diffusion is determined locally and dynamically from the
resolved-scale solution
\cite{Santos-etal:2010,Valli-etal:2018}.
The formulation employs an enriched $P_1$-bubble finite element space, decomposed into a resolved continuous piecewise linear space and an unresolved element-bubble space. Its nonlinear artificial-diffusion coefficient is evaluated from the resolved component, whereas the resulting diffusion operator acts on the full enriched discrete approximation. An important feature of the method is that
the artificial-diffusion contribution is not prescribed through a tuned
stabilization parameter, but is determined from local information
associated with the discrete residual. Wherever this additional contribution vanishes, locally on an element or globally throughout the mesh, the DD formulation coincides with the corresponding Galerkin formulation in the enriched \(P_1\)-bubble space.

A mathematical analysis of a bounded variant of the nonlinear two-scale
DD formulation was developed in \cite{Santos-etal:2021}. In that work,
existence of discrete solutions, stability, and a priori error estimates
were established for stationary advection--diffusion--reaction problems.

The resulting analysis provided an \(O(h^{1/2})\) estimate for a combined quantity involving the energy error and the contribution of the nonlinear artificial-diffusion operator. By nonnegativity of the latter contribution, this estimate also yields an \(O(h^{1/2})\) bound for the energy error, but it does not determine whether the energy error alone converges at a higher rate. Such half-order a priori bounds are also encountered in the analysis of nonlinear shock-capturing discretizations; see, for instance, \cite{Allendes-etal:2017}. In contrast, the numerical experiments in \cite{Santos-etal:2021} consistently displayed first-order convergence in the energy norm, together with optimal convergence rates in the \(H^1\)-seminorm and the \(L^2\)-norm. The sharp a priori convergence rate of the energy error therefore remained unresolved.

The nonlinear dependence of the artificial-diffusion operator on the
discrete solution also raises the question of uniqueness. This issue was
subsequently addressed by Du et al.~\cite{Du-etal:2026}, who proved
existence and uniqueness for the DD formulation analyzed in
\cite{Santos-etal:2021} by means of a contraction-mapping argument for sufficiently fine meshes. Their work also developed a
residual-based a posteriori error estimator and an adaptive linearized DD
strategy. Du et al.~explicitly identified the remaining limitation of
the analysis: ``The proof of the optimal convergence rate is still an open
problem, but it is not a concern in this article''
\cite{Du-etal:2026}. Thus, uniqueness of the discrete DD solution was no longer open, whereas
an optimal a priori convergence result for the bounded formulation of
\cite{Santos-etal:2021} remained unavailable.

In a related contribution, Du, Pan, and Xie
\cite{Du-etal:2026b} introduced a modified artificial-diffusion
coefficient and established existence, 
uniqueness under a mesh condition, and an
optimal first-order estimate for a combined quantity involving the energy
error and a dissipative contribution. Their result confirms that suitable
control of the solution dependence of the nonlinear artificial diffusion
is sufficient to recover optimal convergence. At the same time, these developments motivate a closer examination of the approximation properties of the formulation analyzed in \cite{Santos-etal:2021}. In particular, it is relevant to determine whether the artificial-diffusion operator of that formulation already supports an optimal energy-norm estimate under suitable assumptions.

The present work shows that the DD formulation analyzed in \cite{Santos-etal:2021}, with the same discrete structure and nonlinear artificial-diffusion operator, admits a sharper convergence analysis than previously established. The analysis yields three main advances.

First, we establish a local Lipschitz estimate for the artificial
diffusivity on the resolved-scale space, with a constant of order
$h_T$. The stability properties of the resolved-scale projection
transfer this estimate to the full enriched finite element space. The
resulting bound is then used in a direct comparison of two discrete
solutions to prove uniqueness for sufficiently fine meshes. This
argument differs from the contraction-mapping framework employed in
\cite{Du-etal:2026}, although both approaches yield mesh-dependent
sufficient conditions for uniqueness.

Second, the earlier $O(h^{1/2})$ estimate in
\cite{Santos-etal:2021} concerns a stronger combined quantity involving the
energy error and the square root of the artificial dissipation. This estimate
immediately yields an $O(h^{1/2})$ bound for the energy error, but it does not
establish whether the energy error alone converges at the higher first-order
rate observed numerically. By separating these contributions in the error
analysis, we first prove the optimal first-order estimate
\[
\|u-u_{hb}\|_E=O(h)
\]
for the DD formulation. Using the energy-norm stability of
the resolved-scale projection, the same first-order rate is then shown
to hold for the resolved component $u_h=\kappa_h(u_{hb})$. Third, a
sharper estimate for the nonlinear artificial dissipation is derived,
showing that its square-root measure also decays with first order for fixed
problem data. Consequently, the combined energy-error and
artificial-dissipation measure is also $O(h)$, sharpening the previously
available $O(h^{1/2})$ bound.

The theoretical results are complemented by numerical experiments for
both advection--diffusion and advection--diffusion--reaction problems. A smooth manufactured solution, considered both with and without reaction, confirms the predicted first-order energy-norm convergence while the nonlinear artificial-diffusion operator remains active throughout the refinement sequence. A strongly
advection-dominated problem with sharp outflow layers is also considered
to illustrate the stabilizing effect of the nonlinear diffusion
contribution in a regime with large local P\'eclet numbers. In both tests, we employ the DD formulation without under-relaxation or any other convergence-acceleration mechanism.

The remainder of the paper is organized as follows.
Section~2 introduces the model problem and recalls the DD formulation.
Section~3 establishes the local Lipschitz continuity of the nonlinear
artificial diffusivity. Section~4 presents the 
uniqueness analysis. Section~5 establishes the optimal a priori energy-norm estimate
and the decay of the nonlinear artificial dissipation. Section~6 reports
the numerical experiments, and the final section summarizes the main
conclusions and outlines directions for further investigation.

\section{Model problem and Dynamic Diffusion formulation}
\label{sec:model}

Let $\Omega\subset\mathbb{R}^{d}$, $d\in\{2,3\}$, be a
bounded polygonal domain if
$d=2$, or a bounded polyhedral domain if $d=3$, with Lipschitz
boundary. We consider the steady
advection--diffusion--reaction problem
\begin{equation}
-\varepsilon\Delta u
+\boldsymbol{\beta}\cdot\boldsymbol{\nabla}u
+\sigma u
=
f
\qquad\text{in }\Omega,
\label{eq:model}
\end{equation}
subject to the homogeneous Dirichlet boundary condition
\begin{equation*}
u=0
\qquad\text{on }\partial\Omega,
\end{equation*}
where $\varepsilon>0$ is the diffusion coefficient,
$\boldsymbol{\beta}\in[W^{1,\infty}(\Omega)]^d$ is a divergence-free
velocity field, $\sigma\ge0$ is the
constant reaction coefficient, and
$f\in L^2(\Omega)$ is the prescribed source term.

Let
\[
V:=H_0^1(\Omega),
\]
and define
\begin{equation*}
B(w,v)
:=
\varepsilon
(\boldsymbol{\nabla}w,\boldsymbol{\nabla}v)
+
(\boldsymbol{\beta}\cdot\boldsymbol{\nabla}w,v)
+
\sigma(w,v).
\end{equation*}
Since $\boldsymbol{\nabla}\cdot\boldsymbol{\beta}=0$ and
$v=0$ on $\partial\Omega$, the advective contribution is
skew-symmetric on the diagonal, and hence
\begin{equation*}
B(v,v)
=
\varepsilon |v|_1^2
+
\sigma\|v\|_0^2.
\end{equation*}
Accordingly, we introduce the energy norm
\begin{equation*}
\|v\|_E^2
:=
\varepsilon |v|_1^2
+
\sigma\|v\|_0^2.
\end{equation*}

We use $|\cdot|_1$ to denote the usual $H^1$-seminorm. Recall that,
by the Poincaré inequality, $|\cdot|_1$ defines a norm on
$H_0^1(\Omega)$ equivalent to the full $H^1$-norm.

\subsection{Discrete spaces and scale decomposition}

Let $\mathcal T_h$ be a conforming, shape-regular triangulation of
$\Omega$ into simplices $T$, with
$h_T:=\operatorname{diam}(T)$
and $h:=\max_{T\in\mathcal T_h}h_T$.
We consider the continuous piecewise linear finite element space
\[
V_h^0
:=
\left\{
v_h\in H_0^1(\Omega):
v_h|_T\in\mathbb P_1(T),
\quad T\in\mathcal T_h
\right\},
\]
together with an elementwise bubble space

\[
V_b:=\bigoplus_{T\in\mathcal T_h}V_b(T),
\qquad
V_b(T):=\operatorname{span}\{\psi_T\},
\]
where $\psi_T\in H_0^1(T)$ is the standard simplex bubble of degree
$d+1$, extended by zero outside $T$. The enriched discrete
space is defined by
\begin{equation*}
V_{hb}^0
=
V_h^0\oplus V_b.
\end{equation*}
Thus, every $v_{hb}\in V_{hb}^0$ admits the unique decomposition
\begin{equation*}
v_{hb}=v_h+v_b,
\qquad
v_h\in V_h^0,
\quad
v_b\in V_b.
\end{equation*}

We use the projection introduced in \cite{Guermond:1999}, with the notation adopted in \cite{Santos-etal:2021}. Thus, let
\[
\kappa_h:
V_{hb}^0\longrightarrow V_h^0
\]
denote the projection onto the resolved-scale space, so that
\[
\kappa_h(v_{hb})=v_h,
\qquad
(I-\kappa_h)(v_{hb})=v_b.
\]
For the $P_1$-bubble decomposition considered here, this scale
decomposition is $L^2$-stable~\cite{Guermond:1999}. In particular, there exists a constant
$C_\kappa >0$, independent of $h$, such that
\begin{equation}
\|\kappa_h(v_{hb})\|_{0,T}
\le
C_\kappa\|v_{hb}\|_{0,T},
\qquad
T\in\mathcal T_h.
\label{eq:kappa_stability}
\end{equation}

\subsection{The Dynamic Diffusion formulation}

For $v_h\in V_h^0$, the elementwise strong residual is defined by
\[
 R_T(v_h)
:=
-\varepsilon\Delta v_h
+\boldsymbol{\beta}\cdot\boldsymbol{\nabla}v_h
+\sigma v_h-f.
\]
Since $v_h|_T\in\mathbb P_1(T)$, one has $\Delta v_h=0$ in $T$, and therefore
\begin{equation}\label{residual-local}
R_T(v_h)
:=
\boldsymbol{\beta}\cdot\boldsymbol{\nabla}v_h
+\sigma v_h-f
\qquad\text{in }T.
\end{equation}

The local Péclet number is defined by
\begin{equation*}
Pe_T
=
\frac{
\|\boldsymbol{\beta}\|_{0,\infty,T}h_T
}{
2\varepsilon
},
\end{equation*}
where
$\|\cdot\|_{0,\infty,T}$ denotes the
$W^{0,\infty}(T)=L^\infty(T)$ norm. The choice of this norm  provides the natural local velocity scale and makes $Pe_T$ dimensionless, consistently with its interpretation as the ratio between advective and diffusive
effects.

We introduce the indicator function based on the local Péclet criterion
\begin{equation*}
\chi_T
=
\begin{cases}
1, & Pe_T>1,\\[1mm]
0, & Pe_T\le1.
\end{cases}
\end{equation*}
For $\sigma>0$, the local nonlinear artificial diffusivity is defined
for $v_h\in V_h^0$ by
\begin{equation*}
	\xi_T(v_h)
	=
	\chi_T\,
	\mu(h_T)
	\frac{
		\|R_T(v_h)\|_{0,T}
	}{
		\|v_h\|_{1,T}+\tau
	},
	\qquad
	\tau>0,
\end{equation*}
where $\mu(h_T)>0$ is the local DD length scale and satisfies
\begin{equation*}
	\mu(h_T)
	\le
	C_\mu h_T.
\end{equation*}
When $\sigma=0$, the corresponding definition is
\begin{equation*}
	\xi_T(v_h)
	=
	\chi_T\,
	\mu(h_T)
	\frac{
		\|R_T(v_h)\|_{0,T}
	}{
		|v_h|_{1,T}+\tau
	}.
\end{equation*}
In both cases, the regularization parameter $\tau>0$ prevents the
denominator from vanishing. The local Péclet criterion activates the
artificial diffusivity when $Pe_T>1$ and sets it to zero when
$Pe_T\le1$.

For $z_{hb},u_{hb},v_{hb}\in V_{hb}^0$, let
\[
z_h:=\kappa_h(z_{hb})\in V_h^0.
\]
The local nonlinear artificial-diffusion operator is then defined by
\begin{align*}
	D_T(z_{hb};u_{hb},v_{hb})
	&:=
	\xi_T(z_h)
	\left(
	\boldsymbol{\nabla}u_{hb},
	\boldsymbol{\nabla}v_{hb}
	\right)_T
	\\
	&=
	\xi_T\bigl(\kappa_h(z_{hb})\bigr)
	\left(
	\boldsymbol{\nabla}u_{hb},
	\boldsymbol{\nabla}v_{hb}
	\right)_T.
\end{align*}
Thus, $\xi_T$ is defined on the resolved-scale space $V_h^0$, whereas
the dependence of $D_T$ on its first argument $z_{hb}$ is determined
by the composition $\xi_T\circ\kappa_h$. The corresponding global
operator is defined by
\begin{equation}
	D_h(z_{hb};u_{hb},v_{hb})
	:=
	\sum_{T\in\mathcal T_h}
	D_T(z_{hb};u_{hb},v_{hb}).
	\label{eq:Dh}
\end{equation}

The Dynamic Diffusion approximation is then defined as follows:
find $u_{hb}\in V_{hb}^0$ such that
\begin{equation}
B(u_{hb},v_{hb})
+
D_h(u_{hb};u_{hb},v_{hb})
=
(f,v_{hb}),
\qquad
\forall v_{hb}\in V_{hb}^0.
\label{eq:DD}
\end{equation}

The local Péclet criterion determines whether the additional nonlinear artificial-diffusion contribution is present on a given element. If \(Pe_T\le1\), then \(\xi_T=0\) on \(T\), and the local contribution of \eqref{eq:DD} coincides with that of the Galerkin formulation in the enriched \(P_1\)-bubble space. This local deactivation does not define a different method; the resulting enriched Galerkin contribution corresponds to the zero-artificial-diffusion regime of the DD formulation.

\subsection*{Active and inactive element sets}\label{sec:active-inactive}

For later use in the analysis, we introduce the active and inactive
element sets
\begin{equation*}
	\mathcal T_h^{\mathrm{act}}
	:=
	\left\{
	T\in\mathcal T_h:Pe_T>1
	\right\},
	\qquad
	\mathcal T_h^{\mathrm{inact}}
	:=
	\mathcal T_h\setminus\mathcal T_h^{\mathrm{act}}.
\end{equation*}
Since $\xi_T(v_h)=0$ for every
$T\in\mathcal T_h^{\mathrm{inact}}$, the nonlinear
artificial-diffusion operator defined in \eqref{eq:Dh} can equivalently be written as
\begin{equation}
	D_h(z_{hb};u_{hb},v_{hb})
	=
	\sum_{T\in\mathcal T_h^{\mathrm{act}}}
	D_T(z_{hb};u_{hb},v_{hb}).
	\label{eq:Dh_active}
\end{equation}

\subsection{Basic properties}

The analysis in \cite{Santos-etal:2021} provides the basic properties
of the nonlinear coefficient and of the corresponding discrete
problem. In particular, the artificial diffusivity is nonnegative and
satisfies the elementwise estimate
\begin{equation}
0
\le
\xi_T(v_h)
\le
\mu(h_T)q_T,
\label{eq:xi_bound}
\end{equation}
where
\begin{equation}
q_T
:=
\begin{cases}
\text{}
\|\boldsymbol{\beta}\|_{0,\infty,T}
+\sigma
+\tau^{-1}\|f\|_{0,T},
& \sigma>0,
\\[1mm]
\text{}
\|\boldsymbol{\beta}\|_{0,\infty,T}
+\tau^{-1}\|f\|_{0,T},
& \sigma=0.
\end{cases}
\label{eq:qT}
\end{equation}
Consequently,
\begin{equation*}
	0
	\le
	D_T(z_{hb};u_{hb},u_{hb})
	\le
	\mu(h_T)q_T |u_{hb}|_{1,T}^2.
\end{equation*}
Summing over all elements, the nonlinear artificial-diffusion operator
satisfies
\begin{equation}
	0
	\le
	D_h(z_{hb};u_{hb},u_{hb})
	\le
	\sum_{T\in\mathcal T_h}
	\mu(h_T)q_T|u_{hb}|_{1,T}^2.
	\label{eq:Dh_nonnegative}
\end{equation}

The existence of at least one discrete solution to \eqref{eq:DD} was established in \cite{Santos-etal:2021} and will not be reconsidered here. Moreover, the following lemma establishes mesh-independent stability estimates.
\begin{lemma}[Energy and $H^1$-seminorm stability]
\label{lem:energy_stability}
Every solution $u_{hb}\in V_{hb}^0$ of \eqref{eq:DD} satisfies
\begin{equation}
\|u_{hb}\|_E^2
+
D_h(u_{hb};u_{hb},u_{hb})
\le
C_{\mathrm{st}}\|f\|_0^2,
\label{eq:stability}
\end{equation}
where
\begin{equation*}
C_{\mathrm{st}}
=
\begin{cases}
\displaystyle \frac{1}{\sigma},
& \sigma>0,\\[2mm]
\displaystyle \frac{C_P^2}{\varepsilon},
& \sigma=0,
\end{cases}
\end{equation*}
and $C_P$ denotes the Poincaré constant. In addition, every discrete
solution satisfies
\begin{equation}
|u_{hb}|_1
\le
\frac{\|f\|_0}{2\sqrt{\sigma\varepsilon}},
\qquad \sigma>0,
\label{eq:solution_H1_bound_sigma}
\end{equation}
whereas
\begin{equation}
|u_{hb}|_1
\le
\frac{C_P}{\varepsilon}\|f\|_0,
\qquad \sigma=0.
\label{eq:solution_H1_bound_sigma_zero}
\end{equation}
\end{lemma}

\begin{proof}
Taking $v_{hb}=u_{hb}$ in \eqref{eq:DD} gives
\begin{equation}
	\|u_{hb}\|_E^2
	+
	D_h(u_{hb};u_{hb},u_{hb})
	=
	(f,u_{hb}).
	\label{eq:discrete_energy_identity}
\end{equation}
Assume first that $\sigma>0$. Using the nonnegativity property  \eqref{eq:Dh_nonnegative} and the
Cauchy--Schwarz inequality give
\[
\sigma\|u_{hb}\|_0^2
\le
(f,u_{hb})
\le
\|f\|_0\|u_{hb}\|_0.
\]
Consequently,
\[
\|u_{hb}\|_0
\le
\frac{1}{\sigma}\|f\|_0,
\]
where the conclusion is immediate if $\|u_{hb}\|_0=0$. Returning to the
energy identity, we obtain
\[
\|u_{hb}\|_E^2
+
D_h(u_{hb};u_{hb},u_{hb})
\le
\|f\|_0\|u_{hb}\|_0
\le
\frac{1}{\sigma}\|f\|_0^2.
\]
This proves \eqref{eq:stability} for $\sigma>0$.

The same energy identity also gives
\[
\varepsilon |u_{hb}|_1^2
\le
\|f\|_0\|u_{hb}\|_0
-
\sigma\|u_{hb}\|_0^2.
\]
Using
\[
ax-cx^2
\le
\frac{a^2}{4c},
\qquad c>0,
\]
with $a=\|f\|_0$, $c=\sigma$, and
$x=\|u_{hb}\|_0$, we obtain
\[
\varepsilon |u_{hb}|_1^2
\le
\frac{\|f\|_0^2}{4\sigma},
\]
which proves \eqref{eq:solution_H1_bound_sigma}.

Now let $\sigma=0$. By the Cauchy--Schwarz and Poincaré inequalities,
\[
(f,u_{hb})
\le
C_P\|f\|_0|u_{hb}|_1
=
C_P\varepsilon^{-1/2}\|f\|_0\|u_{hb}\|_E.
\]
Since $\|u_{hb}\|_E=\varepsilon^{1/2}|u_{hb}|_1$ in this case, the
energy identity \eqref{eq:discrete_energy_identity} and \eqref{eq:Dh_nonnegative} also yield
\[
\varepsilon |u_{hb}|_1^2
\le
C_P\|f\|_0|u_{hb}|_1,
\]
which proves \eqref{eq:solution_H1_bound_sigma_zero}.

Young's inequality then yields
\[
(f,u_{hb})
\le
\frac12\|u_{hb}\|_E^2
+
\frac{C_P^2}{2\varepsilon}\|f\|_0^2.
\]
Hence,
\[
\|u_{hb}\|_E^2
+
2D_h(u_{hb};u_{hb},u_{hb})
\le
\frac{C_P^2}{\varepsilon}\|f\|_0^2.
\]
Since $D_h(u_{hb};u_{hb},u_{hb})\ge0$, estimate
\eqref{eq:stability} follows also for $\sigma=0$.
\end{proof}

For $\sigma>0$, the stability estimate in the energy norm is uniform with
respect to both mesh refinement and $\varepsilon$. This does not imply a
uniform bound for the $H^1$ seminorm, since the gradient contribution to the
energy norm is weighted by $\varepsilon$. For $\sigma=0$, the estimate remains
uniform with respect to the mesh size, while its explicit dependence on
$\varepsilon$ arises from the coercivity of the underlying Galerkin bilinear
form. In both cases, the nonnegative of the nonlinear artificial difusion operator  strengthens the
left-hand side of the estimate.

These properties provide the starting point for the analysis developed
below. In particular, the next section establishes a local Lipschitz
estimate for the nonlinear artificial diffusivity, which is subsequently
used to derive a  uniqueness result
 for the DD formulation
considered here.

\begin{remark}[Resolved and bubble components]
	\label{rem:resolved_bubble_roles}
	
	The resolved and bubble components associated with the decomposition
	$V_{hb}^0=V_h^0\oplus V_b$ play different computational roles.
	The degrees of freedom associated with $u_h$ are globally coupled,
	whereas the bubble degrees of freedom are internal to each element and
	are eliminated locally by static condensation. Consequently, the
	assembled global algebraic problem involves only the degrees of freedom
	of $u_h\in V_h^0$, while the bubble component $u_b$ may be recovered
	elementwise after the global problem has been solved.
	
	The theoretical analysis is first carried out for the full DD
	approximation $u_{hb}$. The corresponding convergence result for the
	resolved component $u_h=\kappa_h(u_{hb})$ is subsequently obtained from
	the stability of the projection $\kappa_h$. In the numerical experiments,
	unless otherwise stated, the solution displayed in the plots and used
	to compute the reported approximation errors is $u_h$. The bubble
	component is recovered when quantities involving the full discrete field
	$u_{hb}$ are required.
\end{remark}


\section{Properties of the nonlinear artificial diffusivity}

In this section, we establish an explicit local Lipschitz estimate for
the nonlinear artificial diffusivity $\xi_T$ defining the bounded DD
formulation analyzed in \cite{Santos-etal:2021}. The analysis relies on
stability properties of the resolved-scale projection $\kappa_h$ and on
the local Lipschitz continuity of the residual norm.

Throughout this section,
$w_{hb},u_{hb},v_{hb}\in V_{hb}^0$
denote arbitrary enriched finite element functions and
\[
w_h=\kappa_h(w_{hb}),
\qquad
u_h=\kappa_h(u_{hb}),
\qquad
v_h=\kappa_h(v_{hb}).
\]


\subsection{Elementary estimates}

This subsection collects the auxiliary estimates required for the local
Lipschitz analysis of the artificial diffusivity. The $L^2$-stability
of the resolved-scale projection $\kappa_h$ was recalled in
\eqref{eq:kappa_stability}. Here, we derive stability properties of
$\kappa_h$ in the $H^1$-seminorm, the $H^1$-norm, and the energy norm,
followed by local Lipschitz estimates for the residual norm.

The projection estimates rely on the orthogonality of the resolved and
bubble spaces with respect to the gradient inner product, as shown in
the following lemma.


\begin{lemma}[Local $H^1$-seminorm orthogonality]
\label{lem:orthogonal_decomposition}
Let $v_h\in V_h^0$ and $v_b\in V_b$. Then, for every
$T\in\mathcal T_h$,
\begin{equation*}
\left(
\boldsymbol{\nabla}v_h,
\boldsymbol{\nabla}v_b
\right)_T
=
0.
\end{equation*}
Consequently, for every decomposition
$v_{hb}=v_h+v_b\in V_{hb}^0$,
\begin{equation}
|v_{hb}|_{1,T}^2
=
|v_h|_{1,T}^2
+
|v_b|_{1,T}^2.
\label{eq:orthogonal_decomposition}
\end{equation}
\end{lemma}

\begin{proof}
Since $v_h|_T\in\mathbb P_1(T)$, we have
\[
\Delta v_h=0
\qquad\text{in }T.
\]
Moreover, since $v_b$ is an element bubble function,
\[
v_b=0
\qquad\text{on }\partial T.
\]
Integration by parts therefore gives
\begin{align*}
\left(
\boldsymbol{\nabla}v_h,
\boldsymbol{\nabla}v_b
\right)_T
&=
-
\left(
\Delta v_h,
v_b
\right)_T
+
\int_{\partial T}
v_b\,
\boldsymbol{\nabla}v_h\cdot\mathbf n
\,ds
\\
&=
0.
\end{align*}
Identity~\eqref{eq:orthogonal_decomposition} follows by expanding
\[
|v_h+v_b|_{1,T}^2
=
|v_h|_{1,T}^2
+
2
\left(
\boldsymbol{\nabla}v_h,
\boldsymbol{\nabla}v_b
\right)_T
+
|v_b|_{1,T}^2.
\]
\end{proof}


The $L^2$-stability \eqref{eq:kappa_stability} and
Lemma~\ref{lem:orthogonal_decomposition} imply that $\kappa_h$ is
locally nonexpansive in the $H^1$-seminorm, locally stable in the
$H^1$-norm, and globally stable in the energy norm.

\begin{lemma}[Stability of the resolved-scale projection]
	\label{lem:projection}
	For every $v_{hb}\in V_{hb}^0$ and every
	$T\in\mathcal T_h$,
	\begin{equation}
		|\kappa_h(v_{hb})|_{1,T}
		\le
		|v_{hb}|_{1,T},
		\label{eq:kappa_H1_seminorm}
	\end{equation}
	and
	\begin{equation}
		\|\kappa_h(v_{hb})\|_{1,T}
		\le
		C_{\kappa,1}\|v_{hb}\|_{1,T},
		\label{eq:kappa_H1_norm}
	\end{equation}
	where
	\[
	C_{\kappa,1}:=\max\{C_\kappa,1\}.
	\]
	Moreover,
	\begin{equation}
		\|\kappa_h(v_{hb})\|_E
		\le
		C_{\kappa,E}\|v_{hb}\|_E,
		\qquad
		C_{\kappa,E}:=\max\{1,C_\kappa\}.
		\label{eq:kappa_energy_stability}
	\end{equation}
	In particular, when $\sigma=0$, one may take
	$C_{\kappa,E}=1$.
\end{lemma}
\begin{proof}
	Writing
	\[
	v_{hb}=v_h+v_b,
	\qquad
	v_h=\kappa_h(v_{hb}),
	\qquad
	v_b=(I-\kappa_h)(v_{hb}),
	\]
	identity~\eqref{eq:orthogonal_decomposition} gives
	\[
	|v_{hb}|_{1,T}^2
	=
	|\kappa_h(v_{hb})|_{1,T}^2
	+
	|v_b|_{1,T}^2.
	\]
	Hence,
	\[
	|\kappa_h(v_{hb})|_{1,T}
	\le
	|v_{hb}|_{1,T},
	\]
	which proves \eqref{eq:kappa_H1_seminorm}.
	
	Combining \eqref{eq:kappa_H1_seminorm} with the $L^2$-stability
	\eqref{eq:kappa_stability}, we obtain
	\begin{align*}
		\|\kappa_h(v_{hb})\|_{1,T}^2
		&=
		\|\kappa_h(v_{hb})\|_{0,T}^2
		+
		|\kappa_h(v_{hb})|_{1,T}^2
		\\
		&\le
		C_\kappa^2\|v_{hb}\|_{0,T}^2
		+
		|v_{hb}|_{1,T}^2
		\\
		&\le
		C_{\kappa,1}^2
		\left(
		\|v_{hb}\|_{0,T}^2
		+
		|v_{hb}|_{1,T}^2
		\right)
		\\
		&=
		C_{\kappa,1}^2\|v_{hb}\|_{1,T}^2,
	\end{align*}
	which proves \eqref{eq:kappa_H1_norm}.
	
	Summing the local estimates
	\eqref{eq:kappa_stability} and
	\eqref{eq:kappa_H1_seminorm} over the mesh gives
	\[
	\|\kappa_h(v_{hb})\|_0
	\le
	C_\kappa\|v_{hb}\|_0,
	\qquad
	|\kappa_h(v_{hb})|_1
	\le
	|v_{hb}|_1.
	\]
	Therefore,
	\begin{align*}
		\|\kappa_h(v_{hb})\|_E^2
		&=
		\varepsilon|\kappa_h(v_{hb})|_1^2
		+
		\sigma\|\kappa_h(v_{hb})\|_0^2
		\\
		&\le
		\varepsilon|v_{hb}|_1^2
		+
		\sigma C_\kappa^2\|v_{hb}\|_0^2
		\\
		&\le
		C_{\kappa,E}^2
		\left(
		\varepsilon|v_{hb}|_1^2
		+
		\sigma\|v_{hb}\|_0^2
		\right)
		\\
		&=
		C_{\kappa,E}^2\|v_{hb}\|_E^2,
	\end{align*}
	where
	\[
	C_{\kappa,E}:=\max\{1,C_\kappa\}.
	\]
	Taking square roots proves
	\eqref{eq:kappa_energy_stability}. If $\sigma=0$, the $L^2$ term is
	absent and \eqref{eq:kappa_H1_seminorm} shows that one may take
	$C_{\kappa,E}=1$.
\end{proof}



\begin{lemma}[Local Lipschitz continuity of the residual norm]
	\label{lem:residual}
	For every $u_h,v_h\in V_h^0$ and every
	$T\in\mathcal T_h$,
	\begin{equation}
		\left|
		\|R_T(u_h)\|_{0,T}
		-
		\|R_T(v_h)\|_{0,T}
		\right|
		\le
		\|R_T(u_h)-R_T(v_h)\|_{0,T}
		\le
		a_T\|u_h-v_h\|_{1,T},
		\label{eq:residual_lipschitz}
	\end{equation}
	where
	\begin{equation}
		a_T
		:=
		\left(
		\|\boldsymbol{\beta}\|_{0,\infty,T}^2+\sigma^2
		\right)^{1/2}.
		\label{eq:aT}
	\end{equation}
\end{lemma}

\begin{proof}
	The first inequality in \eqref{eq:residual_lipschitz} follows from the
	reverse triangle inequality. By the residual definition \eqref{residual-local},
	\[
	R_T(u_h)-R_T(v_h)
	=
	\boldsymbol{\beta}\cdot
	\boldsymbol{\nabla}(u_h-v_h)
	+
	\sigma(u_h-v_h).
	\]
Therefore, using the Cauchy--Schwarz inequality in $\mathbb R^2$, we
obtain
\begin{align*}
	\|R_T(u_h)-R_T(v_h)\|_{0,T}
	&\le
	\|\boldsymbol{\beta}\|_{0,\infty,T}
	|u_h-v_h|_{1,T}
	+
	\sigma\|u_h-v_h\|_{0,T}
	\\
	&\le
	\left(
	\|\boldsymbol{\beta}\|_{0,\infty,T}^2+\sigma^2
	\right)^{1/2}
	\left(
	|u_h-v_h|_{1,T}^2
	+
	\|u_h-v_h\|_{0,T}^2
	\right)^{1/2}
	\\
	&=
	a_T\|u_h-v_h\|_{1,T}.
\end{align*}
	This proves \eqref{eq:residual_lipschitz}.
\end{proof}


To establish the Lipschitz continuity of \(\xi_T\), we need to control the variation of the inverse regularized denominator. The following elementary estimate provides this control.

\begin{lemma}[Elementwise Lipschitz continuity of the denominator]
	\label{lem:denominator}
	Define
	\[
	\mathcal N_T(u_h)
	:=
	\begin{cases}
		\|u_h\|_{1,T},
		& \sigma>0,
		\\[1mm]
		|u_h|_{1,T},
		& \sigma=0.
	\end{cases}
	\]
	Then, for every $u_h,v_h\in V_h^0$,
	\begin{align}
		\left|
		\bigl(\mathcal N_T(u_h)+\tau\bigr)
		-
		\bigl(\mathcal N_T(v_h)+\tau\bigr)
		\right|
		&=
		|\mathcal N_T(u_h)-\mathcal N_T(v_h)|
		\nonumber\\
		&\le
		\mathcal N_T(u_h-v_h)
		\le
		\|u_h-v_h\|_{1,T}.
		\label{eq:denominator_lipschitz}
	\end{align}
	When $\sigma=0$, the intermediate bound gives the sharper estimate
	\[
	|\mathcal N_T(u_h)-\mathcal N_T(v_h)|
	\le
	|u_h-v_h|_{1,T}.
	\]
\end{lemma}

\begin{proof}
	The result follows directly from the reverse triangle inequality for
	the $H^1$ norm when $\sigma>0$ and for the $H^1$ seminorm when
	$\sigma=0$.
\end{proof}

\subsection{Local Lipschitz continuity}

The estimates underlying the Lipschitz continuity of the artificial
diffusivity appeared within the continuity argument used in the
existence proof of \cite{Santos-etal:2021}. However, this property was
not stated as a separate local result or employed in a uniqueness
analysis.

We first establish the Lipschitz continuity of $\xi_T$ on the
resolved-scale space $V_h^0$. Since the first argument of the nonlinear
artificial-diffusion operator enters through the composition
$\xi_T\circ\kappa_h$, the corresponding result on the enriched space
$V_{hb}^0$ then follows from the stability properties of $\kappa_h$
established in Lemma~\ref{lem:projection}.

\begin{theorem}[Local Lipschitz continuity of the artificial diffusivity]
	\label{thm:lipschitz}
	For every $u_h,v_h\in V_h^0$ and every $T\in\mathcal T_h$,
	\begin{equation}
		\left|
		\xi_T(u_h)-\xi_T(v_h)
		\right|
		\le
		\widehat L_T
		\|u_h-v_h\|_{1,T},
		\label{eq:xi_lipschitz_resolved}
	\end{equation}
	where
	\begin{equation}
		\widehat L_T
		:=
		\chi_T
		\frac{\mu(h_T)}{\tau}
		\left(
		a_T+q_T
		\right).
		\label{eq:LT_resolved}
	\end{equation}
	Consequently, for every $u_{hb},v_{hb}\in V_{hb}^0$,
	\begin{equation}
		\left|
		\xi_T\bigl(\kappa_h(u_{hb})\bigr)
		-
		\xi_T\bigl(\kappa_h(v_{hb})\bigr)
		\right|
		\le
		L_T
		\|u_{hb}-v_{hb}\|_{1,T},
		\label{eq:xi_lipschitz}
	\end{equation}
	where
	\begin{equation}
		L_T
		:=
		C_{\kappa,1}\widehat L_T
		=
		\chi_T
		\frac{\mu(h_T)}{\tau}
		C_{\kappa,1}
		\left(
		a_T+q_T
		\right).
		\label{eq:LT}
	\end{equation}
	Here, $a_T$ and $q_T$ are defined in
	\eqref{eq:aT} and \eqref{eq:qT}, respectively, and
	$C_{\kappa,1}$ is the constant appearing in
	Lemma~\ref{lem:projection}.
	
	Moreover, for fixed problem data and fixed $\tau>0$, there exists a
	constant $C_L>0$, independent of $h$, such that
	\begin{equation}
		L_T\le C_Lh_T,
		\qquad
		T\in\mathcal T_h.
		\label{eq:LT_order}
	\end{equation}
	When $\sigma=0$, estimates
	\eqref{eq:xi_lipschitz_resolved} and \eqref{eq:xi_lipschitz}
	can be sharpened to
	\begin{equation}
		\left|
		\xi_T(u_h)-\xi_T(v_h)
		\right|
		\le
		L_T^0|u_h-v_h|_{1,T},
		\label{eq:xi_lipschitz_resolved_sigma_zero}
	\end{equation}
	and
	\begin{equation}
		\left|
		\xi_T\bigl(\kappa_h(u_{hb})\bigr)
		-
		\xi_T\bigl(\kappa_h(v_{hb})\bigr)
		\right|
		\le
		L_T^0|u_{hb}-v_{hb}|_{1,T},
		\label{eq:xi_lipschitz_sigma_zero}
	\end{equation}
	respectively, where
	\begin{equation}
		L_T^0
		:=
		\chi_T
		\frac{\mu(h_T)}{\tau}
		\left(
		\|\boldsymbol{\beta}\|_{0,\infty,T}
		+
		q_T
		\right).
		\label{eq:LT_sigma_zero}
	\end{equation}
	There also exists a constant $C_L^0>0$, independent of $h$, such that
	\begin{equation}
		L_T^0\le C_L^0h_T,
		\qquad
		T\in\mathcal T_h.
		\label{eq:LT_sigma_zero_order}
	\end{equation}
\end{theorem}

\begin{proof}
	Let $u_h,v_h\in V_h^0$. If
	$T\in\mathcal T_h^{\mathrm{inact}}$, then $\chi_T=0$, and hence
	\[
	\xi_T(u_h)=\xi_T(v_h)=0.
	\]
	Thus, all the stated estimates hold immediately on inactive elements.
	
	Let $T\in\mathcal T_h^{\mathrm{act}}$. By
	Lemma~\ref{lem:residual},
	\[
	\left|
	\|R_T(u_h)\|_{0,T}
	-
	\|R_T(v_h)\|_{0,T}
	\right|
	\le
	a_T\|u_h-v_h\|_{1,T}.
	\]
	Moreover, \eqref{eq:denominator_lipschitz} gives
	\[
	\left|
	\mathcal N_T(u_h)-\mathcal N_T(v_h)
	\right|
	\le
	\|u_h-v_h\|_{1,T}.
	\]
	Furthermore, the definition of $\xi_T$ and the bound
	\eqref{eq:xi_bound} give
	\begin{equation}
		\frac{\|R_T(v_h)\|_{0,T}}
		{\mathcal N_T(v_h)+\tau}
		\le
		q_T.
		\label{eq:R_div_N}
	\end{equation}
	
	Using the identity
	\[
	\frac{a}{b}-\frac{c}{d}
	=
	\frac{a-c}{b}
	+
	\frac{c(d-b)}{bd},
	\]
	we obtain
	\begin{align*}
		\left|
		\xi_T(u_h)-\xi_T(v_h)
		\right|
		&\le
		\mu(h_T)
		\frac{
			\left|
			\|R_T(u_h)\|_{0,T}
			-
			\|R_T(v_h)\|_{0,T}
			\right|
		}
		{\mathcal N_T(u_h)+\tau}
		\\
		&\quad
		+
		\mu(h_T)
		\frac{
			\|R_T(v_h)\|_{0,T}
			\left|
			\mathcal N_T(u_h)-\mathcal N_T(v_h)
			\right|
		}
		{
			\bigl(\mathcal N_T(u_h)+\tau\bigr)
			\bigl(\mathcal N_T(v_h)+\tau\bigr)
		}.
	\end{align*}
	Since
	\[
	\mathcal N_T(u_h)+\tau\ge\tau,
	\]
	the first term satisfies
	\[
	\mu(h_T)
	\frac{
		\left|
		\|R_T(u_h)\|_{0,T}
		-
		\|R_T(v_h)\|_{0,T}
		\right|
	}
	{\mathcal N_T(u_h)+\tau}
	\le
	\frac{\mu(h_T)}{\tau}
	a_T\|u_h-v_h\|_{1,T}.
	\]
	For the second term, using \eqref{eq:R_div_N} and again
	$\mathcal N_T(u_h)+\tau\ge\tau$, we obtain
	\[
	\mu(h_T)
	\frac{
		\|R_T(v_h)\|_{0,T}
		\left|
		\mathcal N_T(u_h)-\mathcal N_T(v_h)
		\right|
	}
	{
		\bigl(\mathcal N_T(u_h)+\tau\bigr)
		\bigl(\mathcal N_T(v_h)+\tau\bigr)
	}
	\le
	\frac{\mu(h_T)}{\tau}
	q_T\|u_h-v_h\|_{1,T}.
	\]
	Combining these two estimates yields
	\[
	\left|
	\xi_T(u_h)-\xi_T(v_h)
	\right|
	\le
	\frac{\mu(h_T)}{\tau}
	\left(
	a_T+q_T
	\right)
	\|u_h-v_h\|_{1,T},
	\]
	which proves \eqref{eq:xi_lipschitz_resolved}, with
	$\widehat L_T$ defined in \eqref{eq:LT_resolved}.
	
	Now let $u_{hb},v_{hb}\in V_{hb}^0$ and set
	\[
	u_h=\kappa_h(u_{hb}),
	\qquad
	v_h=\kappa_h(v_{hb}).
	\]
	Applying \eqref{eq:xi_lipschitz_resolved}, the linearity of
	$\kappa_h$, and its local $H^1$-stability gives
	\begin{align*}
		&\left|
		\xi_T\bigl(\kappa_h(u_{hb})\bigr)
		-
		\xi_T\bigl(\kappa_h(v_{hb})\bigr)
		\right|
		\\
		&\qquad\le
		\widehat L_T
		\|\kappa_h(u_{hb})-\kappa_h(v_{hb})\|_{1,T}
		\\
		&\qquad=
		\widehat L_T
		\|\kappa_h(u_{hb}-v_{hb})\|_{1,T}
		\\
		&\qquad\le
		C_{\kappa,1}\widehat L_T
		\|u_{hb}-v_{hb}\|_{1,T}.
	\end{align*}
	This proves \eqref{eq:xi_lipschitz}. In view of
	\eqref{eq:LT_resolved}, the corresponding Lipschitz constant is
	$L_T=C_{\kappa,1}\widehat L_T$, as stated in \eqref{eq:LT}.

	Since
	\[
	\mu(h_T)\le C_\mu h_T,
	\qquad
	\|f\|_{0,T}\le\|f\|_0,
	\]
	we may take
	\[
	C_L
	:=
	\frac{C_\mu C_{\kappa,1}}{\tau}
	\left[
	\left(
	\|\boldsymbol{\beta}\|_{0,\infty}^2+\sigma^2
	\right)^{1/2}
	+
	\|\boldsymbol{\beta}\|_{0,\infty}
	+
	\sigma
	+
	\tau^{-1}\|f\|_0
	\right].
	\]
	It follows that
	\[
	L_T\le C_Lh_T,
	\]
	which proves \eqref{eq:LT_order}.
	
	Finally, suppose that $\sigma=0$. In this case,
	Lemmas~\ref{lem:residual} and~\ref{lem:denominator} yield the sharper
	estimates
	\[
	\left|
	\|R_T(u_h)\|_{0,T}
	-
	\|R_T(v_h)\|_{0,T}
	\right|
	\le
	\|\boldsymbol{\beta}\|_{0,\infty,T}
	|u_h-v_h|_{1,T}
	\]
	and
	\[
	\left|
	\mathcal N_T(u_h)-\mathcal N_T(v_h)
	\right|
	\le
	|u_h-v_h|_{1,T}.
	\]
	Repeating the preceding quotient estimate gives
	\[
	\left|
	\xi_T(u_h)-\xi_T(v_h)
	\right|
	\le
	\frac{\mu(h_T)}{\tau}
	\left(
	\|\boldsymbol{\beta}\|_{0,\infty,T}
	+
	q_T
	\right)
	|u_h-v_h|_{1,T},
	\]
	which proves \eqref{eq:xi_lipschitz_resolved_sigma_zero}.
	
	For $u_{hb},v_{hb}\in V_{hb}^0$, the local nonexpansiveness of
	$\kappa_h$ in the $H^1$-seminorm gives
	\[
	|\kappa_h(u_{hb}-v_{hb})|_{1,T}
	\le
	|u_{hb}-v_{hb}|_{1,T}.
	\]
	Therefore, the same constant \(L_T^0\) applies to the composition
	$\xi_T\circ\kappa_h$, proving
	\eqref{eq:xi_lipschitz_sigma_zero}.
	
	Finally, using $\mu(h_T)\le C_\mu h_T$ and
	$\|f\|_{0,T}\le\|f\|_0$, we may take
	\[
	C_L^0
	:=
	\frac{C_\mu}{\tau}
	\left(
	2\|\boldsymbol{\beta}\|_{0,\infty}
	+
	\tau^{-1}\|f\|_0
	\right),
	\]
	which proves \eqref{eq:LT_sigma_zero_order}.
\end{proof}

\begin{remark}
	When $\sigma=0$, the passage from the resolved-scale space $V_h^0$
	to the enriched space $V_{hb}^0$ does not increase the local Lipschitz
	constant. Indeed, the projection $\kappa_h$ is locally nonexpansive in
	the $H^1$-seminorm, and therefore
	\[
	|\kappa_h(u_{hb}-v_{hb})|_{1,T}
	\le
	|u_{hb}-v_{hb}|_{1,T}.
	\]
	Consequently, the same constant $L_T^0$ controls both $\xi_T$ on
	$V_h^0$ and the composition $\xi_T\circ\kappa_h$ on $V_{hb}^0$.
	In contrast, for $\sigma>0$, the estimate is expressed in the full
	$H^1$-norm and the transfer to the enriched space introduces the
	projection-stability factor $C_{\kappa,1}$.
\end{remark}

Theorem~\ref{thm:lipschitz} first establishes the local Lipschitz
continuity of the artificial diffusivity on the resolved-scale space.
The stability of the resolved-scale projection then transfers this
property to the enriched space, with a local Lipschitz constant of
order $h_T$. This estimate is the main ingredient in the uniqueness
analysis of the next section.


\section{Uniqueness of the Dynamic Diffusion method}
\label{sec:uniqueness}

Existence of at least one discrete solution of \eqref{eq:DD} was
established in \cite{Santos-etal:2021}. We now prove uniqueness by
combining the local Lipschitz estimates of
Theorem~\ref{thm:lipschitz} with the coercivity and stability
properties of the DD formulation.

A recent analysis of the DD method obtained existence and uniqueness
through a contraction-mapping argument \cite{Du-etal:2026}. The
argument developed here proceeds directly by comparing two discrete
solutions. Subtraction of the corresponding discrete equations,
followed by control of the nonlinear artificial-diffusion term, yields
mesh-dependent sufficient conditions for uniqueness.

If $f=0$, testing \eqref{eq:DD} with the discrete solution and using
the coercivity of $B$ and the nonnegativity of $D_h$ show that the
solution is identically zero. We therefore consider $f\neq0$ in what
follows.

\begin{theorem}[Uniqueness for $\sigma>0$]
	\label{thm:uniqueness_sigma}
	Assume $\sigma>0$, and let
	$u_{hb},v_{hb}\in V_{hb}^0$
	be two solutions of \eqref{eq:DD}. Define
	\begin{equation*}
		L_h
		:=
		\max_{T\in\mathcal T_h}L_T,
	\end{equation*}
	where $L_T$ is the local Lipschitz constant defined in
	\eqref{eq:LT}. Let $C_0>0$ be such that
	\begin{equation}
		\|w\|_1
		\le
		C_0|w|_1,
		\qquad
		\forall w\in H_0^1(\Omega),
		\label{eq:poincare_full}
	\end{equation}
	as follows from the Poincaré inequality. If
	\begin{equation}
		\frac{C_0L_h\|f\|_0}
		{2\sqrt{\sigma}\,\varepsilon^{3/2}}
		<1,
		\label{eq:uniqueness_condition_sigma}
	\end{equation}
	then
	\[
	u_{hb}=v_{hb}.
	\]
	Hence, the discrete solution of \eqref{eq:DD} is unique.
\end{theorem}

\begin{proof}
	
	Let
	\[
	e:=u_{hb}-v_{hb}\in V_{hb}^0.
	\]
	Subtracting the discrete equations satisfied by $u_{hb}$ and
	$v_{hb}$ and choosing $e$ as test function gives
	\begin{equation}
		B(e,e)
		+
		D_h(u_{hb};u_{hb},e)
		-
		D_h(v_{hb};v_{hb},e)
		=
		0.
		\label{eq:uniqueness_identity}
	\end{equation}
	Adding and subtracting $D_h(u_{hb};v_{hb},e)$, we obtain
	\begin{align}
		&D_h(u_{hb};u_{hb},e)
		-
		D_h(v_{hb};v_{hb},e)
		=
		D_h(u_{hb};e,e)
		+
		D_h(u_{hb};v_{hb},e)
		-
		D_h(v_{hb};v_{hb},e).
		\label{eq:uniqueness_decomposition}
	\end{align}
	Since
	\[
	D_h(u_{hb};e,e)\ge0
	\]
	and
	\[
	B(e,e)
	=
	\varepsilon|e|_1^2+\sigma\|e\|_0^2
	\ge
	\varepsilon|e|_1^2,
	\]
	it follows from
	\eqref{eq:uniqueness_identity}--\eqref{eq:uniqueness_decomposition}
	that
	\begin{equation}
		\varepsilon|e|_1^2
		\le
		\left|
		D_h(u_{hb};v_{hb},e)
		-
		D_h(v_{hb};v_{hb},e)
		\right|.
		\label{eq:uniqueness_coercive_bound}
	\end{equation}
	Using the active-element representation \eqref{eq:Dh_active}, the
	right-hand side satisfies
	\begin{align}
		&\left|
		D_h(u_{hb};v_{hb},e)
		-
		D_h(v_{hb};v_{hb},e)
		\right|
		\le
		\sum_{T\in\mathcal T_h^{\mathrm{act}}}
		\left|
		\xi_T\bigl(\kappa_h(u_{hb})\bigr)
		-
		\xi_T\bigl(\kappa_h(v_{hb})\bigr)
		\right|
		|v_{hb}|_{1,T}|e|_{1,T}.
		\label{eq:nonlinear_difference_active}
	\end{align}	
	Applying \eqref{eq:xi_lipschitz} and the definition of $L_h$, we obtain
	\begin{equation}
		\varepsilon|e|_1^2
		\le
		L_h
		\sum_{T\in\mathcal T_h^{\mathrm{act}}}
		\|e\|_{1,T}|v_{hb}|_{1,T}|e|_{1,T}.
		\label{eq:uniqueness_lipschitz_bound}
	\end{equation}
	By the Cauchy--Schwarz inequality,
	\begin{align*}
		&\sum_{T\in\mathcal T_h^{\mathrm{act}}}
		\|e\|_{1,T}|v_{hb}|_{1,T}|e|_{1,T}
		\le
		\left(
		\sum_{T\in\mathcal T_h^{\mathrm{act}}}
		|v_{hb}|_{1,T}^2
		\right)^{1/2}
		\left(
		\sum_{T\in\mathcal T_h^{\mathrm{act}}}
		\|e\|_{1,T}^2|e|_{1,T}^2
		\right)^{1/2}.
	\end{align*}
	The first factor is bounded by $|v_{hb}|_1$. For the second factor,
	the nonnegativity of the local contributions gives
	\begin{align*}
		\sum_{T\in\mathcal T_h^{\mathrm{act}}}
		\|e\|_{1,T}^2|e|_{1,T}^2
		&\le
		\left(
		\sum_{T\in\mathcal T_h^{\mathrm{act}}}
		\|e\|_{1,T}^2
		\right)
		\left(
		\sum_{T\in\mathcal T_h^{\mathrm{act}}}
		|e|_{1,T}^2
		\right)
		\\
		&\le
		\|e\|_1^2|e|_1^2.
	\end{align*}
	Therefore,
	\begin{equation}
		\sum_{T\in\mathcal T_h^{\mathrm{act}}}
		\|e\|_{1,T}|v_{hb}|_{1,T}|e|_{1,T}
		\le
		|v_{hb}|_1\|e\|_1|e|_1.
		\label{eq:uniqueness_sum_bound}
	\end{equation}
	Combining \eqref{eq:uniqueness_lipschitz_bound},
	\eqref{eq:uniqueness_sum_bound}, and \eqref{eq:poincare_full}, we find
	\begin{equation}
		\varepsilon|e|_1^2
		\le
		C_0L_h|v_{hb}|_1|e|_1^2.
		\label{eq:uniqueness_bound_3}
	\end{equation}
Substituting the stability estimate
\eqref{eq:solution_H1_bound_sigma} from
Lemma~\ref{lem:energy_stability} into
\eqref{eq:uniqueness_bound_3} yields
\[
\varepsilon|e|_1^2
\le
\frac{C_0L_h\|f\|_0}
{2\sqrt{\sigma\varepsilon}}
|e|_1^2.
\]
	Dividing both sides by $\varepsilon>0$ and collecting the terms gives
	\begin{equation}
		\left(
		1-
		\frac{C_0L_h\|f\|_0}
		{2\sqrt{\sigma}\,\varepsilon^{3/2}}
		\right)
		|e|_1^2
		\le0.
		\label{eq:uniqueness_final_sigma}
	\end{equation}
By \eqref{eq:uniqueness_condition_sigma}, the factor in parentheses in
\eqref{eq:uniqueness_final_sigma} is
strictly positive. Since $|e|_1^2\ge0$, it follows directly that
$|e|_1=0$. Moreover, $e\in H_0^1(\Omega)$, so that $e=0$ and,
consequently,
\[
u_{hb}=v_{hb}.
\]
\end{proof}

When $\sigma=0$, the sharper seminorm-based Lipschitz estimate
\eqref{eq:xi_lipschitz_sigma_zero}, together with the stability bound
\eqref{eq:solution_H1_bound_sigma_zero}, yields the corresponding
uniqueness criterion. The algebraic part of the argument is the same
as in the case $\sigma>0$; only the estimates specific to the
reaction-free problem are detailed below.

\begin{theorem}[Uniqueness for $\sigma=0$]
	\label{thm:uniqueness_sigma_zero}
	Assume $\sigma=0$, and let
	$u_{hb},v_{hb}\in V_{hb}^0$
	be two solutions of \eqref{eq:DD}. Define
	\[
	L_h^0
	:=
	\max_{T\in\mathcal T_h}L_T^0,
	\]
where $L_T^0$ is defined in \eqref{eq:LT_sigma_zero} and $C_P$ is the
Poincaré constant introduced in Lemma~\ref{lem:energy_stability}.
	If
	\begin{equation}
		\frac{C_PL_h^0\|f\|_0}{\varepsilon^2}
		<1,
		\label{eq:uniqueness_condition_sigma_zero}
	\end{equation}
	then
	\[
	u_{hb}=v_{hb}.
	\]
	Hence, the discrete solution of \eqref{eq:DD} is unique.
\end{theorem}

\begin{proof}
	Let
	\[
	e:=u_{hb}-v_{hb}\in V_{hb}^0.
	\]
	Repeating the algebraic argument in
	\eqref{eq:uniqueness_identity}--\eqref{eq:uniqueness_decomposition},
	and observing that
	\[
	B(e,e)=\varepsilon|e|_1^2
	\]
	we recover estimate
	\eqref{eq:uniqueness_coercive_bound}.
	Using \eqref{eq:nonlinear_difference_active} and the sharper Lipschitz
	estimate \eqref{eq:xi_lipschitz_sigma_zero}, we find
	\begin{align*}
		\varepsilon|e|_1^2
		&\le
		L_h^0
		\sum_{T\in\mathcal T_h^{\mathrm{act}}}
		|e|_{1,T}|v_{hb}|_{1,T}|e|_{1,T}
		\\
		&=
		L_h^0
		\sum_{T\in\mathcal T_h^{\mathrm{act}}}
		|v_{hb}|_{1,T}|e|_{1,T}^2.
	\end{align*}
	By the Cauchy--Schwarz inequality,
	\begin{align*}
		\sum_{T\in\mathcal T_h^{\mathrm{act}}}
		|v_{hb}|_{1,T}|e|_{1,T}^2
		&\le
		\left(
		\sum_{T\in\mathcal T_h^{\mathrm{act}}}
		|v_{hb}|_{1,T}^2
		\right)^{1/2}
		\left(
		\sum_{T\in\mathcal T_h^{\mathrm{act}}}
		|e|_{1,T}^4
		\right)^{1/2}
		\\
		&\le
		|v_{hb}|_1|e|_1^2.
	\end{align*}
	Therefore,
	\begin{equation}
		\varepsilon|e|_1^2
		\le
		L_h^0|v_{hb}|_1|e|_1^2.
		\label{eq:uniqueness_bound_sigma_zero}
	\end{equation}
	Substituting the stability estimate
	\eqref{eq:solution_H1_bound_sigma_zero} from
	Lemma~\ref{lem:energy_stability} into
	\eqref{eq:uniqueness_bound_sigma_zero} yields
	\[
	\varepsilon|e|_1^2
	\le
	\frac{C_PL_h^0\|f\|_0}{\varepsilon}
	|e|_1^2.
	\]
	Dividing by $\varepsilon>0$ and collecting the terms gives
	\[
	\left(
	1-
	\frac{C_PL_h^0\|f\|_0}{\varepsilon^2}
	\right)
	|e|_1^2
	\le0.
	\]
	By \eqref{eq:uniqueness_condition_sigma_zero}, the factor in
	parentheses is strictly positive. Since $|e|_1^2\ge0$, it follows
	directly that $|e|_1=0$. Moreover, $e\in H_0^1(\Omega)$, so that
	$e=0$ and, consequently,
	\[
	u_{hb}=v_{hb}.
	\]
\end{proof}

\begin{remark}
	Since
	\[
	L_h\le C_Lh
	\qquad\text{and}\qquad
	L_h^0\le C_L^0h,
	\]
	the sufficient conditions in
	Theorems~\ref{thm:uniqueness_sigma} and
	\ref{thm:uniqueness_sigma_zero} are satisfied for sufficiently small
	$h$. Therefore, for each fixed $\varepsilon>0$, the DD discrete
	solution is unique on sufficiently fine meshes. The mesh-resolution
	threshold obtained from these conditions depends on $\varepsilon$
	because the proofs rely on diffusion-weighted coercivity estimates.
	Accordingly, the resulting uniqueness criteria are not uniform as
	$\varepsilon\to0$.
\end{remark}

\begin{corollary}[Uniqueness of the resolved approximation]
	\label{cor:uniqueness_resolved}
	Under the assumptions of either
	Theorem~\ref{thm:uniqueness_sigma} or
	Theorem~\ref{thm:uniqueness_sigma_zero}, the resolved component
	\[
	u_h=\kappa_h(u_{hb})\in V_h^0
	\]
	is uniquely determined.
\end{corollary}

\begin{proof}
	The result follows immediately from the uniqueness of $u_{hb}$ and
	the identity $u_h=\kappa_h(u_{hb})$.
\end{proof}

With the solvability properties of the discrete problem established,
we next examine the accuracy of the DD approximation. In particular,
we prove its optimal-order convergence in the energy norm.

\section{Convergence analysis}
\label{sec:error_analysis}

We now establish optimal first-order convergence of the DD approximation
in the energy norm. We first analyze the enriched approximation
$u_{hb}$ and subsequently derive the corresponding result for the
resolved finite element solution
$u_h=\kappa_h(u_{hb})$.

The key point is to estimate the energy error separately from the
nonlinear artificial-dissipation measure. In contrast with the combined
estimate obtained in \cite{Santos-etal:2021}, the argument developed
below treats the nonlinear artificial-diffusion contribution directly
as a bounded perturbation in the error equation.

The analysis does not require uniqueness. All the estimates below hold
for any discrete solution of \eqref{eq:DD} satisfying the stability
bounds of Lemma~\ref{lem:energy_stability}.

\subsection{Optimal convergence in the energy norm}
\label{sec:optimal_energy}

Let $\Pi_hu\in V_h^0$ denote the finite element interpolant of $u$.
For any discrete solution $u_{hb}\in V_{hb}^0$, we decompose the error
as
\begin{equation}
	u-u_{hb}
	=
	\eta+\varphi,
	\qquad
	\eta:=u-\Pi_hu,
	\qquad
	\varphi:=\Pi_hu-u_{hb}.
	\label{eq:error_decomposition}
\end{equation}
For $u\in H^2(\Omega)$, the standard interpolation estimates give
\begin{equation}
	\|\eta\|_0
	\le
	C_Ih^2|u|_2,
	\qquad
	|\eta|_1
	\le
	C_Ih|u|_2.
	\label{eq:interp}
\end{equation}

We further define
\begin{equation}
	\delta_h
	:=
	\max_{T\in\mathcal T_h}
	\mu(h_T)q_T.
	\label{eq:delta_h}
\end{equation}
Since $\mu(h_T)\le C_\mu h_T$ and
$\|f\|_{0,T}\le\|f\|_0$, the constant
\[
C_\delta
:=
C_\mu
\left(
\|\boldsymbol{\beta}\|_{0,\infty}
+
\sigma
+
\tau^{-1}\|f\|_0
\right)
\]
is independent of $h$ and satisfies
\begin{equation}
	\delta_h\le C_\delta h.
	\label{eq:delta_order}
\end{equation}

We next derive two relations that will be used in the convergence
proofs for both $\sigma>0$ and $\sigma=0$. Since $u$ satisfies the continuous problem and $u_{hb}$
satisfies \eqref{eq:DD}, testing the corresponding error equation with
$\varphi\in V_{hb}^0$ gives
\begin{equation}
	\|\varphi\|_E^2
	=
	-B(\eta,\varphi)
	+
	D_h(u_{hb};u_{hb},\varphi).
	\label{eq:phi_identity}
\end{equation}
Moreover, since the nonlinear artificial diffusivity vanishes on
inactive elements, \eqref{eq:xi_bound}, the definition
\eqref{eq:delta_h}, and the Cauchy--Schwarz inequality yield
\begin{align*}
	\left|
	D_h(u_{hb};u_{hb},\varphi)
	\right|
	&\le
	\sum_{T\in\mathcal T_h^{\mathrm{act}}}
	\mu(h_T)q_T
	|u_{hb}|_{1,T}|\varphi|_{1,T}
	\\
	&\le
	\delta_h
	\sum_{T\in\mathcal T_h^{\mathrm{act}}}
	|u_{hb}|_{1,T}|\varphi|_{1,T}
	\\
	&\le
	\delta_h|u_{hb}|_1|\varphi|_1.
\end{align*}
Using
\[
|\varphi|_1
\le
\varepsilon^{-1/2}\|\varphi\|_E,
\]
we obtain the common estimate
\begin{equation}
	\left|
	D_h(u_{hb};u_{hb},\varphi)
	\right|
	\le
	\delta_h\varepsilon^{-1/2}
	|u_{hb}|_1\|\varphi\|_E.
	\label{eq:DD_error_common}
\end{equation}


\begin{theorem}[Optimal energy-norm estimate for $\sigma>0$]
	\label{thm:optimal_sigma}
	Assume $\sigma>0$, and let
	$u\in H^2(\Omega)\cap H_0^1(\Omega)$ be the solution of
	\eqref{eq:model}. Let $u_{hb}\in V_{hb}^0$ be any solution of the DD
	problem \eqref{eq:DD}. Then
	\begin{align}
		\|u-u_{hb}\|_E
		\le
		C\Bigg[
		&
		\left(
		\varepsilon^{1/2}
		+
		\sigma^{-1/2}
		\|\boldsymbol{\beta}\|_{0,\infty}
		\right)h
		+
		\sigma^{1/2}h^2
		\Bigg]|u|_2
		+
		C
		\frac{\delta_h}
		{\sqrt{\sigma}\,\varepsilon}
		\|f\|_0,
		\label{eq:optimal_sigma}
	\end{align}
	where $C>0$ is independent of $h$. Consequently, for fixed
	$\varepsilon$, $\sigma$, $\boldsymbol{\beta}$, $f$, and $\tau$,
	\begin{equation}
		\|u-u_{hb}\|_E
		=
		O(h).
		\label{eq:optimal_rate_sigma}
	\end{equation}
\end{theorem}

\begin{proof}
	We first estimate the interpolation term in
	\eqref{eq:phi_identity}. By the definition of $B$,
	\begin{align}
		|B(\eta,\varphi)|
		\le{}&
		\varepsilon|\eta|_1|\varphi|_1
		+
		\|\boldsymbol{\beta}\|_{0,\infty}
		|\eta|_1\|\varphi\|_0
		+
		\sigma\|\eta\|_0\|\varphi\|_0
		\nonumber\\
		\le{}&
		\left[
		\varepsilon^{1/2}|\eta|_1
		+
		\sigma^{-1/2}
		\|\boldsymbol{\beta}\|_{0,\infty}|\eta|_1
		+
		\sigma^{1/2}\|\eta\|_0
		\right]
		\|\varphi\|_E,
		\label{eq:B_eta_phi_sigma}
	\end{align}
	where we used
	\[
	|\varphi|_1
	\le
	\varepsilon^{-1/2}\|\varphi\|_E,
	\qquad
	\|\varphi\|_0
	\le
	\sigma^{-1/2}\|\varphi\|_E.
	\]
Combining the stability estimate
\eqref{eq:solution_H1_bound_sigma} with
\eqref{eq:DD_error_common}, we obtain
\begin{equation}
	\left|
	D_h(u_{hb};u_{hb},\varphi)
	\right|
	\le
	\frac{\delta_h}
	{2\sqrt{\sigma}\,\varepsilon}
	\|f\|_0\|\varphi\|_E.
	\label{eq:DD_error_bound_sigma}
\end{equation}
	Substituting \eqref{eq:B_eta_phi_sigma} and
	\eqref{eq:DD_error_bound_sigma} into \eqref{eq:phi_identity} yields
	\begin{align*}
		\|\varphi\|_E^2
		\le
		\Bigg[
		&
		\left(
		\varepsilon^{1/2}
		+
		\sigma^{-1/2}
		\|\boldsymbol{\beta}\|_{0,\infty}
		\right)
		|\eta|_1
		+
		\sigma^{1/2}\|\eta\|_0
		+
		\frac{\delta_h}
		{2\sqrt{\sigma}\,\varepsilon}
		\|f\|_0
		\Bigg]
		\|\varphi\|_E.
	\end{align*}
	If $\|\varphi\|_E=0$, the desired estimate follows immediately.
	Otherwise, division by $\|\varphi\|_E$ gives
	\begin{align}
		\|\varphi\|_E
		\le{}&
		\left(
		\varepsilon^{1/2}
		+
		\sigma^{-1/2}
		\|\boldsymbol{\beta}\|_{0,\infty}
		\right)
		|\eta|_1
		+
		\sigma^{1/2}\|\eta\|_0
	+
		\frac{\delta_h}
		{2\sqrt{\sigma}\,\varepsilon}
		\|f\|_0.
		\label{eq:phi_bound_sigma}
	\end{align}
	Moreover,
	\begin{equation}
		\|\eta\|_E
		\le
		\varepsilon^{1/2}|\eta|_1
		+
		\sigma^{1/2}\|\eta\|_0.
		\label{eq:eta_energy_sigma}
	\end{equation}
	Therefore, the triangle inequality gives
	\[
	\|u-u_{hb}\|_E
	\le
	\|\eta\|_E+\|\varphi\|_E.
	\]
	Using \eqref{eq:interp}, \eqref{eq:phi_bound_sigma}, and
	\eqref{eq:eta_energy_sigma}, we obtain
	\begin{align*}
		\|u-u_{hb}\|_E
		\le
		C\Bigg[
		&
		\left(
		\varepsilon^{1/2}
		+
		\sigma^{-1/2}
		\|\boldsymbol{\beta}\|_{0,\infty}
		\right)h
		+
		\sigma^{1/2}h^2
		\Bigg]|u|_2
	+
		C
		\frac{\delta_h}
		{\sqrt{\sigma}\,\varepsilon}
		\|f\|_0,
	\end{align*}
	which proves \eqref{eq:optimal_sigma}. Finally,
	\eqref{eq:delta_order} implies, for fixed
	$\varepsilon$, $\sigma$, $\boldsymbol{\beta}$, $f$, and $\tau$, that
	\[
	\|u-u_{hb}\|_E=O(h).
	\]
	This proves \eqref{eq:optimal_rate_sigma}.
\end{proof}


When $\sigma=0$, the energy norm contains only the diffusion-weighted
$H^1$-seminorm. The advective interpolation term is therefore treated
by integration by parts, which transfers the derivative from the
interpolation error to the test function and allows the
$O(h^2)$ estimate for $\|\eta\|_0$ to be used. Combined with
\eqref{eq:DD_error_common} and the stability estimate for the
reaction-free problem, this yields the optimal energy-norm rate.

\begin{theorem}[Optimal energy-norm estimate for $\sigma=0$]
	\label{thm:optimal_zero}
	Assume $\sigma=0$, and let
	$u\in H^2(\Omega)\cap H_0^1(\Omega)$ be the solution of
	\eqref{eq:model}. Let $u_{hb}\in V_{hb}^0$ be any solution of the DD
	problem \eqref{eq:DD}. Then
	\begin{align}
		\|u-u_{hb}\|_E
		\le{}&
		C\left(
		\varepsilon^{1/2}h
		+
		\varepsilon^{-1/2}
		\|\boldsymbol{\beta}\|_{0,\infty}h^2
		\right)|u|_2
		+
		C\delta_h\varepsilon^{-3/2}\|f\|_0,
		\label{eq:optimal_zero}
	\end{align}
	where $C>0$ is independent of $h$. Consequently, for fixed
	$\varepsilon$, $\boldsymbol{\beta}$, $f$, and $\tau$,
	\begin{equation}
		\|u-u_{hb}\|_E
		=
		O(h).
		\label{eq:optimal_rate_zero}
	\end{equation}
\end{theorem}

\begin{proof}
	When $\sigma=0$, the energy norm reduces to
	\[
	\|v\|_E
	=
	\varepsilon^{1/2}|v|_1.
	\]
	We first estimate the interpolation term in
	\eqref{eq:phi_identity}. Its diffusive contribution satisfies
	\[
	\varepsilon
	\left|
	(\boldsymbol{\nabla}\eta,
	\boldsymbol{\nabla}\varphi)
	\right|
	\le
	\varepsilon^{1/2}
	|\eta|_1
	\|\varphi\|_E.
	\]
	For the advective contribution, using
	$\boldsymbol{\nabla}\cdot\boldsymbol{\beta}=0$ and the homogeneous
	boundary conditions, integration by parts gives
	\[
	(\boldsymbol{\beta}\cdot\boldsymbol{\nabla}\eta,\varphi)
	=
	-
	(\eta,\boldsymbol{\beta}\cdot\boldsymbol{\nabla}\varphi).
	\]
	Therefore,
	\begin{align*}
		\left|
		(\boldsymbol{\beta}\cdot\boldsymbol{\nabla}\eta,\varphi)
		\right|
		&\le
		\|\boldsymbol{\beta}\|_{0,\infty}
		\|\eta\|_0
		|\varphi|_1
		\\
		&=
		\varepsilon^{-1/2}
		\|\boldsymbol{\beta}\|_{0,\infty}
		\|\eta\|_0
		\|\varphi\|_E.
	\end{align*}
	Combining the diffusive and advective estimates above, we obtain
	\begin{equation}
		|B(\eta,\varphi)|
		\le
		\left[
		\varepsilon^{1/2}|\eta|_1
		+
		\varepsilon^{-1/2}
		\|\boldsymbol{\beta}\|_{0,\infty}
		\|\eta\|_0
		\right]
		\|\varphi\|_E.
		\label{eq:B_eta_phi_zero}
	\end{equation}
Applying the stability estimate
\eqref{eq:solution_H1_bound_sigma_zero} to
\eqref{eq:DD_error_common}, we get
	\begin{equation}
		\left|
		D_h(u_{hb};u_{hb},\varphi)
		\right|
		\le
		C_P\delta_h\varepsilon^{-3/2}
		\|f\|_0
		\|\varphi\|_E.
		\label{eq:DD_error_zero}
	\end{equation}
	Substituting \eqref{eq:B_eta_phi_zero} and
	\eqref{eq:DD_error_zero} into \eqref{eq:phi_identity} yields
	\begin{align*}
		\|\varphi\|_E^2
		\le
		\Big[
		&
		\varepsilon^{1/2}|\eta|_1
		+
		\varepsilon^{-1/2}
		\|\boldsymbol{\beta}\|_{0,\infty}
		\|\eta\|_0
		+
		C_P\delta_h\varepsilon^{-3/2}
		\|f\|_0
		\Big]
		\|\varphi\|_E.
	\end{align*}
	If $\|\varphi\|_E=0$, the desired estimate follows immediately.
	Otherwise, dividing by $\|\varphi\|_E$ gives
	\begin{align}
		\|\varphi\|_E
		\le{}&
		\varepsilon^{1/2}|\eta|_1
		+
		\varepsilon^{-1/2}
		\|\boldsymbol{\beta}\|_{0,\infty}
		\|\eta\|_0
		+
		C_P\delta_h\varepsilon^{-3/2}
		\|f\|_0.
		\label{eq:phi_bound_zero}
	\end{align}
	Moreover, since $\sigma=0$,
	\[
	\|\eta\|_E
	=
	\varepsilon^{1/2}|\eta|_1.
	\]
	The triangle inequality therefore gives
	\[
	\|u-u_{hb}\|_E
	\le
	\|\eta\|_E+\|\varphi\|_E.
	\]
	Using \eqref{eq:interp} and \eqref{eq:phi_bound_zero}, we obtain
	\begin{align*}
		\|u-u_{hb}\|_E
		\le{}&
		C\left(
		\varepsilon^{1/2}h
		+
		\varepsilon^{-1/2}
		\|\boldsymbol{\beta}\|_{0,\infty}h^2
		\right)|u|_2
		+
		C\delta_h\varepsilon^{-3/2}\|f\|_0,
	\end{align*}
	where the Poincaré constant has been absorbed into the generic constant
	$C$. This proves \eqref{eq:optimal_zero}. Finally, using
	\eqref{eq:delta_order}, we conclude that, for fixed
	$\varepsilon$, $\boldsymbol{\beta}$, $f$, and $\tau$,
	\[
	\|u-u_{hb}\|_E
	=
	O(h),
	\]
	which proves \eqref{eq:optimal_rate_zero}.
\end{proof}


\begin{remark}
	Theorems~\ref{thm:optimal_sigma} and~\ref{thm:optimal_zero} establish
	optimal first-order energy-norm convergence of the DD approximation
	with respect to mesh refinement. The constants in the estimates are
	independent of $h$ but depend explicitly on $\varepsilon$. In the
	reaction case, the constant also depends on $\sigma$. Moreover, in the
	presence of layers, the regularity factor $|u|_2$ may itself depend on
	$\varepsilon$. Consequently, the results establish the optimal rate for
	each fixed $\varepsilon>0$, but do not provide an estimate that is
	uniform in the singularly perturbed limit $\varepsilon\to0$.
\end{remark}
\subsection{Decay of the nonlinear artificial dissipation}
\label{sec:artificial_dissipation_decay}

The previous results establish first-order convergence of the DD
approximation in the energy norm. We now investigate separately the
decay of the nonlinear artificial dissipation generated by the DD
operator.

Let $u_h=\kappa_h(u_{hb})$ be the resolved component of the enriched
DD solution. We define
\begin{equation}
	A_h
	:=
	D_h(u_{hb};u_{hb},u_{hb})
	=
	\sum_{T\in\mathcal T_h^{\mathrm{act}}}
	\xi_T(u_h)|u_{hb}|_{1,T}^2.
	\label{eq:Ah_definition}
\end{equation}
By \eqref{eq:Dh_nonnegative}, we have $A_h\ge0$. We
therefore introduce the associated dissipation measure
\begin{equation}
	S_h:=A_h^{1/2}.
	\label{eq:Sh_definition}
\end{equation}
The aim of this subsection is to prove
\[
A_h=O(h^2),
\qquad\text{or, equivalently,}\qquad
S_h=O(h).
\]
This result shows that the nonlinear artificial diffusion vanishes at
the same asymptotic rate as the energy error under mesh refinement.

\begin{lemma}[Global residual estimate]
	\label{lem:global_residual}
	Let $u_{hb}\in V_{hb}^0$ be a solution of the DD problem, and let
	$u_h=\kappa_h(u_{hb})$. Define the broken global residual norm by
	\[
	\|R(u_h)\|_0^2
	:=
	\sum_{T\in\mathcal T_h}
	\|R_T(u_h)\|_{0,T}^2.
	\]
	Then
	\begin{equation}
		\|R(u_h)\|_0
		\le
		C_R\|f\|_0,
		\label{eq:global_residual}
	\end{equation}
	where
	\begin{equation}
		C_R
		=
		\begin{cases}
			\displaystyle
			1+C_\kappa
			+
			\frac{\|\boldsymbol{\beta}\|_{0,\infty}}
			{2\sqrt{\sigma\varepsilon}},
			& \sigma>0,
			\\[3mm]
			\displaystyle
			1+
			\frac{C_P}{\varepsilon}
			\|\boldsymbol{\beta}\|_{0,\infty},
			& \sigma=0.
		\end{cases}
		\label{eq:CR}
	\end{equation}
	In both cases, $C_R$ is independent of $h$.
\end{lemma}

\begin{proof}
	By the definition of the residual and the triangle inequality,
	\begin{equation}
		\|R(u_h)\|_0
		\le
		\|\boldsymbol{\beta}\|_{0,\infty}|u_h|_1
		+
		\sigma\|u_h\|_0
		+
		\|f\|_0.
		\label{eq:global_residual_intermediate}
	\end{equation}
	Assume first that $\sigma>0$. Combining estimate \eqref{eq:kappa_H1_seminorm} with \eqref{eq:solution_H1_bound_sigma}, we obtain
	\[
	|u_h|_1
	\le
	\frac{\|f\|_0}
	{2\sqrt{\sigma\varepsilon}}.
	\]
Moreover, for $\sigma>0$, estimate \eqref{eq:stability} gives
\[
\sigma\|u_{hb}\|_0^2
\le
\|u_{hb}\|_E^2
\le
\frac{1}{\sigma}\|f\|_0^2.
\]
Therefore,
\[
\sigma\|u_{hb}\|_0
\le
\|f\|_0.
\]
	The $L^2$-stability \eqref{eq:kappa_stability} therefore gives
	\[
	\sigma\|u_h\|_0
	\le
	C_\kappa\|f\|_0.
	\]
	Substituting these estimates into
	\eqref{eq:global_residual_intermediate}, we find
	\begin{equation}\label{estim-R1}
	\|R(u_h)\|_0
	\le
	\left(
	1+C_\kappa
	+
	\frac{\|\boldsymbol{\beta}\|_{0,\infty}}
	{2\sqrt{\sigma\varepsilon}}
	\right)
	\|f\|_0,
	\end{equation}

	Now assume that $\sigma=0$. By
	\eqref{eq:kappa_H1_seminorm} and
	\eqref{eq:solution_H1_bound_sigma_zero},
	\[
	|u_h|_1
	\le
	\frac{C_P}{\varepsilon}\|f\|_0.
	\]
	Since the reaction term is absent,
	\eqref{eq:global_residual_intermediate} reduces to
	\[
	\|R(u_h)\|_0
	\le
	\|\boldsymbol{\beta}\|_{0,\infty}|u_h|_1
	+
	\|f\|_0.
	\]
	Consequently,
	\begin{equation}\label{estim-R2}
	\|R(u_h)\|_0
	\le
	\left(
	1+
	\frac{C_P}{\varepsilon}
	\|\boldsymbol{\beta}\|_{0,\infty}
	\right)
	\|f\|_0.
	\end{equation}
	The two estimates \eqref{estim-R1} and \eqref{estim-R2} prove \eqref{eq:global_residual} with $C_R$ given by
	\eqref{eq:CR}.
\end{proof}

We introduce the auxiliary quantity
\begin{equation}
	C_h
	:=
	D_h(u_{hb};u_{hb},\Pi_hu).
	\label{eq:Ch_definition}
\end{equation}
By the error decomposition \eqref{eq:error_decomposition},
\[
u_{hb}
=
\Pi_hu-\varphi.
\]
Since $D_h(u_{hb};u_{hb},\cdot)$ is linear in its third argument,
\eqref{eq:Ah_definition} and \eqref{eq:Ch_definition} give
\[
A_h
=
C_h-D_h(u_{hb};u_{hb},\varphi).
\]
Substituting this identity into \eqref{eq:phi_identity}, we obtain
\begin{equation}
	\|\varphi\|_E^2+A_h
	=
	-B(\eta,\varphi)+C_h.
	\label{eq:error_dissipation_identity}
\end{equation}

By the orthogonality established in
Lemma~\ref{lem:orthogonal_decomposition} and the fact that
$\Pi_hu\in V_h^0$, the bubble component of $u_{hb}$ does not
contribute to $C_h$. Therefore,
\[
C_h
=
\sum_{T\in\mathcal T_h^{\mathrm{act}}}
\xi_T(u_h)
\left(
\boldsymbol{\nabla}u_h,
\boldsymbol{\nabla}\Pi_hu
\right)_T.
\]
Writing
\[
\Pi_hu=(\Pi_hu-u)+u,
\]
we decompose
\begin{equation}
	C_h=C_{h,1}+C_{h,2},
	\label{eq:Ch_decomposition}
\end{equation}
where
\begin{align*}
	C_{h,1}
	&:=
	\sum_{T\in\mathcal T_h^{\mathrm{act}}}
	\xi_T(u_h)
	\left(
	\boldsymbol{\nabla}u_h,
	\boldsymbol{\nabla}(\Pi_hu-u)
	\right)_T,
	\\
	C_{h,2}
	&:=
	\sum_{T\in\mathcal T_h^{\mathrm{act}}}
	\xi_T(u_h)
	\left(
	\boldsymbol{\nabla}u_h,
	\boldsymbol{\nabla}u
	\right)_T.
\end{align*}

\begin{lemma}[Bounds for the auxiliary diffusion terms]
	\label{lem:Ch1}
	\label{lem:Ch2}
	For every $\theta>0$,
	\begin{align}
		|C_{h,1}|
		&\le
		\theta A_h
		+
		C\theta^{-1}h^3|u|_2^2,
		\label{eq:Ch1_estimate_compact}
		\\
		|C_{h,2}|
		&\le
		C\varepsilon^{-1}h^2
		\|\boldsymbol{\beta}\|_{0,\infty}
		C_R\|f\|_0|u|_1,
		\label{eq:Ch2_final_compact}
	\end{align}
	where $C>0$ is independent of $h$, and $C_R$ is given in
	Lemma~\ref{lem:global_residual}.
\end{lemma}

\begin{proof}
	By the definition of $C_{h,1}$ and the weighted Cauchy--Schwarz
	inequality,
	\begin{align*}
		|C_{h,1}|
		&\le
		\left(
		\sum_{T\in\mathcal T_h^{\mathrm{act}}}
		\xi_T(u_h)|u_h|_{1,T}^2
		\right)^{1/2}
		\left(
		\sum_{T\in\mathcal T_h^{\mathrm{act}}}
		\xi_T(u_h)|\Pi_hu-u|_{1,T}^2
		\right)^{1/2}.
	\end{align*}
	Since $u_h=\kappa_h(u_{hb})$, the estimate
	\eqref{eq:kappa_H1_seminorm} gives
	\[
	|u_h|_{1,T}
	\le
	|u_{hb}|_{1,T}.
	\]
	Using also $\xi_T(u_h)\ge0$ and \eqref{eq:Ah_definition}, we obtain
	\[
	\sum_{T\in\mathcal T_h^{\mathrm{act}}}
	\xi_T(u_h)|u_h|_{1,T}^2
	\le
	A_h.
	\]
	For the second factor, the bound \eqref{eq:xi_bound},
	the estimate $\mu(h_T)\le C_\mu h_T$, and the local interpolation
	estimate give
	\begin{align*}
		\sum_{T\in\mathcal T_h^{\mathrm{act}}}
		\xi_T(u_h)|\Pi_hu-u|_{1,T}^2
		&\le
		C
		\sum_{T\in\mathcal T_h^{\mathrm{act}}}
		h_T|\Pi_hu-u|_{1,T}^2
		\\
		&\le
		C
		\sum_{T\in\mathcal T_h^{\mathrm{act}}}
		h_T^3|u|_{2,T}^2
		\\
		&\le
		Ch^3|u|_2^2.
	\end{align*}
	Consequently,
	\[
	|C_{h,1}|
	\le
	Ch^{3/2}A_h^{1/2}|u|_2.
	\]
	Young's inequality then yields, for every $\theta>0$,
	\[
	|C_{h,1}|
	\le
	\theta A_h
	+
	C\theta^{-1}h^3|u|_2^2,
	\]
	which proves \eqref{eq:Ch1_estimate_compact}.
	
	We next estimate $C_{h,2}$. Only active elements contribute to this
	term. For either $\sigma>0$ or $\sigma=0$, the denominator in the
	definition of $\xi_T(u_h)$ satisfies
	\[
	\mathcal N_T(u_h)+\tau
	\ge
	|u_h|_{1,T}.
	\]
	Hence,
	\begin{equation}
		\xi_T(u_h)|u_h|_{1,T}
		\le
		\mu(h_T)\|R_T(u_h)\|_{0,T}.
		\label{eq:xi_gradient_residual}
	\end{equation}
	Moreover, if $T\in\mathcal T_h^{\mathrm{act}}$, then $Pe_T>1$, and
	therefore
	\[
	1
	<
	\frac{
		\|\boldsymbol{\beta}\|_{0,\infty,T}h_T
	}
	{2\varepsilon}.
	\]
	It follows that
	\[
	h_T
	<
	\frac{
		\|\boldsymbol{\beta}\|_{0,\infty,T}h_T^2
	}
	{2\varepsilon}.
	\]
	Combining this inequality with
	$\mu(h_T)\le C_\mu h_T$, we obtain
	\begin{equation}
		\mu(h_T)
		\le
		C\varepsilon^{-1}
		\|\boldsymbol{\beta}\|_{0,\infty,T}h_T^2,
		\qquad
		T\in\mathcal T_h^{\mathrm{act}}.
		\label{eq:mu_active_bound}
	\end{equation}
	Using \eqref{eq:xi_gradient_residual} and
	\eqref{eq:mu_active_bound} in the definition of $C_{h,2}$ gives
	\begin{align*}
		|C_{h,2}|
		&\le
		\sum_{T\in\mathcal T_h^{\mathrm{act}}}
		\xi_T(u_h)|u_h|_{1,T}|u|_{1,T}
		\\
		&\le
		\sum_{T\in\mathcal T_h^{\mathrm{act}}}
		\mu(h_T)
		\|R_T(u_h)\|_{0,T}|u|_{1,T}
		\\
		&\le
		C\varepsilon^{-1}
		\sum_{T\in\mathcal T_h^{\mathrm{act}}}
		\|\boldsymbol{\beta}\|_{0,\infty,T}
		h_T^2
		\|R_T(u_h)\|_{0,T}|u|_{1,T}
		\\
		&\le
		C\varepsilon^{-1}h^2
		\|\boldsymbol{\beta}\|_{0,\infty}
		\sum_{T\in\mathcal T_h^{\mathrm{act}}}
		\|R_T(u_h)\|_{0,T}|u|_{1,T}.
	\end{align*}
	Applying the Cauchy--Schwarz inequality over the active elements, we
	find
	\begin{align*}
		|C_{h,2}|
		&\le
		C\varepsilon^{-1}h^2
		\|\boldsymbol{\beta}\|_{0,\infty}
		\left(
		\sum_{T\in\mathcal T_h^{\mathrm{act}}}
		\|R_T(u_h)\|_{0,T}^2
		\right)^{1/2}
		\left(
		\sum_{T\in\mathcal T_h^{\mathrm{act}}}
		|u|_{1,T}^2
		\right)^{1/2}
		\\
		&\le
		C\varepsilon^{-1}h^2
		\|\boldsymbol{\beta}\|_{0,\infty}
		\|R(u_h)\|_0|u|_1.
	\end{align*}
	Finally, applying the global residual estimate
	\eqref{eq:global_residual}, we obtain
	\[
	|C_{h,2}|
	\le
	C\varepsilon^{-1}h^2
	\|\boldsymbol{\beta}\|_{0,\infty}
	C_R\|f\|_0|u|_1,
	\]
	which proves \eqref{eq:Ch2_final_compact}.
\end{proof}

\begin{theorem}[First-order decay of the nonlinear artificial dissipation]
	\label{thm:Ah_decay}
	Let $u\in H^2(\Omega)\cap H_0^1(\Omega)$ be the solution of
	\eqref{eq:model}, and let $u_{hb}\in V_{hb}^0$ be any solution of the
	DD problem \eqref{eq:DD}. Then, for fixed problem data,
	\begin{equation}
		A_h
		\le
		Ch^2,
		\label{eq:Ah_decay}
	\end{equation}
	where $C>0$ is independent of $h$. Consequently,
	\begin{equation}
		S_h=A_h^{1/2}=O(h).
		\label{eq:Sh_rate}
	\end{equation}
	The result holds for both $\sigma>0$ and $\sigma=0$.
\end{theorem}

\begin{proof}
	It follows from \eqref{eq:error_dissipation_identity} and
	\eqref{eq:Ch_decomposition} that
	\begin{equation}
		\|\varphi\|_E^2+A_h
		\le
		|B(\eta,\varphi)|
		+
		|C_{h,1}|
		+
		|C_{h,2}|.
		\label{eq:Ah_basic_bound}
	\end{equation}
	Assume first that $\sigma>0$. Combining
	\eqref{eq:B_eta_phi_sigma} with the interpolation estimates
	\eqref{eq:interp}, we obtain
	\[
	|B(\eta,\varphi)|
	\le
	C
	\left[
	\left(
	\varepsilon^{1/2}
	+
	\sigma^{-1/2}
	\|\boldsymbol{\beta}\|_{0,\infty}
	\right)h
	+
	\sigma^{1/2}h^2
	\right]
	|u|_2\|\varphi\|_E.
	\]
	Hence, Young's inequality gives, for every $\gamma>0$,
	\begin{align}
		|B(\eta,\varphi)|
		\le{}&
		\gamma\|\varphi\|_E^2
		+
		C\gamma^{-1}
		\left[
		\left(
		\varepsilon
		+
		\sigma^{-1}
		\|\boldsymbol{\beta}\|_{0,\infty}^2
		\right)h^2
		+
		\sigma h^4
		\right]
		|u|_2^2.
		\label{eq:B_eta_Young_sigma}
	\end{align}
	
	Now assume that $\sigma=0$. Combining
	\eqref{eq:B_eta_phi_zero} with \eqref{eq:interp}, we find
	\[
	|B(\eta,\varphi)|
	\le
	C
	\left(
	\varepsilon^{1/2}h
	+
	\varepsilon^{-1/2}
	\|\boldsymbol{\beta}\|_{0,\infty}h^2
	\right)
	|u|_2\|\varphi\|_E.
	\]
	Another application of Young's inequality gives
	\begin{align}
		|B(\eta,\varphi)|
		\le{}&
		\gamma\|\varphi\|_E^2
		+
		C\gamma^{-1}
		\left(
		\varepsilon h^2
		+
		\varepsilon^{-1}
		\|\boldsymbol{\beta}\|_{0,\infty}^2h^4
		\right)
		|u|_2^2.
		\label{eq:B_eta_Young_zero}
	\end{align}
 Applying \eqref{eq:Ch1_estimate_compact} and
 \eqref{eq:Ch2_final_compact} to \eqref{eq:Ah_basic_bound}, and using
 either \eqref{eq:B_eta_Young_sigma} or
 \eqref{eq:B_eta_Young_zero}, according to the value of $\sigma$, we
 obtain, for every $\gamma,\theta>0$,
	\[
	\|\varphi\|_E^2+A_h
	\le
	\gamma\|\varphi\|_E^2
	+
	\theta A_h
	+
	C\left(h^2+h^3+h^4\right),
	\]
	where $C>0$ may depend on the fixed problem data but is independent of
	$h$.
	Choosing
	\[
	\gamma=\theta=\frac14
	\]
	and absorbing the corresponding terms into the left-hand side yields
	\[
	\frac34\|\varphi\|_E^2
	+
	\frac34A_h
	\le
	C\left(h^2+h^3+h^4\right).
	\]
	For $h\le1$,
	\[
	h^3\le h^2,
	\qquad
	h^4\le h^2,
	\]
	and hence
	\[
	\|\varphi\|_E^2+A_h
	\le
	Ch^2.
	\]
	In particular,
	\[
	A_h\le Ch^2.
	\]
	Hence,
	\[
	S_h = A_h^{1/2} \le Ch,
	\]
	which proves \eqref{eq:Ah_decay} and \eqref{eq:Sh_rate}.
\end{proof}

\begin{remark}
	Theorem~\ref{thm:Ah_decay} describes the decay of the nonlinear
	artificial dissipation without requiring the artificial-diffusion
	operator $D_h$ to have already vanished. Along a refinement sequence for which active elements are
	still present, the associated dissipation measure satisfies
	\[
	S_h=A_h^{1/2}=O(h).
	\]
	This behavior is consistent with the local Péclet switch. Indeed, for
	fixed $\varepsilon>0$ and bounded $\boldsymbol{\beta}$, sufficiently
	fine meshes satisfy $Pe_T\le1$ on every element. In that regime,
	$\xi_T=0$ throughout the mesh, and hence $A_h=0$ exactly. Thus, the
	estimate characterizes the decay of the artificial dissipation in the
	active regime and continues to hold after the artificial-diffusion operator has been completely deactivated.
\end{remark}

\subsection{Energy error and artificial dissipation}
\label{sec:energy_error_dissipation}

The analysis in \cite{Santos-etal:2021} yielded an $O(h^{1/2})$
estimate for a quantity involving both the energy error and the
nonlinear artificial dissipation. The separate first-order estimates
established in the preceding subsections show that this stronger
quantity also converges with order $O(h)$.

\begin{theorem}[Energy-error and artificial-dissipation estimate]
	\label{thm:combined_hb}
	Let $u\in H^2(\Omega)\cap H_0^1(\Omega)$ be the solution of
	\eqref{eq:model}, and let $u_{hb}\in V_{hb}^0$ be any solution of the
	DD problem \eqref{eq:DD}. Then, for fixed problem data,
	\begin{equation}
		\left(
		\|u-u_{hb}\|_E^2+A_h
		\right)^{1/2}
		\le
		Ch,
		\label{eq:combined_hb_rate}
	\end{equation}
	where $C>0$ is independent of $h$. The result holds for both
	$\sigma>0$ and $\sigma=0$.
\end{theorem}

\begin{proof}
	By Theorems~\ref{thm:optimal_sigma} and
	\ref{thm:optimal_zero}, there exists a constant $C_E>0$, independent
	of $h$, such that
	\[
	\|u-u_{hb}\|_E
	\le
	C_Eh
	\]
	for both $\sigma>0$ and $\sigma=0$. Moreover,
	Theorem~\ref{thm:Ah_decay} gives
	\[
	A_h^{1/2}
	\le
	C_Ah,
	\]
	where $C_A>0$ is independent of $h$. Therefore,
	\[
	\left(
	\|u-u_{hb}\|_E^2+A_h
	\right)^{1/2}
	\le
	\|u-u_{hb}\|_E+A_h^{1/2}
	\le
	(C_E+C_A)h,
	\]
	which proves \eqref{eq:combined_hb_rate}.
\end{proof}


\begin{corollary}[Convergence of the resolved approximation]
	\label{cor:resolved_scale_convergence}
	Let the assumptions of Theorem~\ref{thm:combined_hb} hold, and let
	\[
	u_h:=\kappa_h(u_{hb})\in V_h^0
	\]
	be the resolved component of the DD approximation. Then
	\begin{equation}
		\|u-u_h\|_E
		\le
		(1+C_{\kappa,E})
		\|u-\Pi_hu\|_E
		+
		C_{\kappa,E}
		\|u-u_{hb}\|_E,
		\label{eq:resolved_error_bound}
	\end{equation}
	where $C_{\kappa,E}$ is the energy-norm stability constant in
	\eqref{eq:kappa_energy_stability}. Consequently, for fixed problem
	data,
	\begin{equation}
		\|u-u_h\|_E
		\le
		Ch,
		\label{eq:resolved_energy_rate}
	\end{equation}
	and
	\begin{equation}
		\left(
		\|u-u_h\|_E^2+A_h
		\right)^{1/2}
		\le
		Ch,
		\label{eq:resolved_combined_rate}
	\end{equation}
	where $C>0$ is independent of $h$. The result holds for both
	$\sigma>0$ and $\sigma=0$.
\end{corollary}

\begin{proof}
	Since $\Pi_hu\in V_h^0$ and $\kappa_h$ is the identity on $V_h^0$,
	\[
	\kappa_h(\Pi_hu)=\Pi_hu.
	\]
	The linearity and energy-norm stability of $\kappa_h$ therefore give
	\begin{align*}
		\|\Pi_hu-u_h\|_E
		&=
		\|\kappa_h(\Pi_hu-u_{hb})\|_E
		\\
		&\le
		C_{\kappa,E}
		\|\Pi_hu-u_{hb}\|_E.
	\end{align*}
	Hence, by the triangle inequality,
	\begin{align*}
		\|u-u_h\|_E
		&\le
		\|u-\Pi_hu\|_E
		+
		\|\Pi_hu-u_h\|_E
		\\
		&\le
		\|u-\Pi_hu\|_E
		+
		C_{\kappa,E}
		\|\Pi_hu-u_{hb}\|_E
		\\
		&\le
		(1+C_{\kappa,E})
		\|u-\Pi_hu\|_E
		+
		C_{\kappa,E}
		\|u-u_{hb}\|_E,
	\end{align*}
	which proves \eqref{eq:resolved_error_bound}.
	
	The interpolation estimates \eqref{eq:interp} imply
	\[
	\|u-\Pi_hu\|_E
	\le
	Ch.
	\]
	Moreover, Theorem~\ref{thm:combined_hb} gives
	\[
	\|u-u_{hb}\|_E
	\le
	Ch
	\qquad\text{and}\qquad
	A_h^{1/2}\le Ch.
	\]
	Substituting the first estimate into
	\eqref{eq:resolved_error_bound} proves
	\eqref{eq:resolved_energy_rate}. Finally,
	\[
	\left(
	\|u-u_h\|_E^2+A_h
	\right)^{1/2}
	\le
	\|u-u_h\|_E+A_h^{1/2}
	\le
	Ch,
	\]
	which proves \eqref{eq:resolved_combined_rate}.
\end{proof}

\begin{remark}
	\label{rem:resolved_scale_interpretation}
	Estimate \eqref{eq:resolved_energy_rate} shows that the projection onto
	the resolved space preserves the optimal energy-norm convergence rate
	of the enriched DD approximation. This has direct computational
	relevance because $u_h$ contains the globally coupled degrees of
	freedom, whereas the bubble component is eliminated locally by static
	condensation and recovered elementwise when required. Thus, static
	condensation does not degrade the asymptotic accuracy of the DD
	approximation.
	
	In addition, estimate \eqref{eq:resolved_combined_rate} shows that the
	resolved-scale energy error and the square-root artificial-dissipation
	measure $A_h^{1/2}$ both converge with first order.
\end{remark}

\section{Numerical experiments}
\label{sec:numerics}

The numerical experiments have two complementary purposes. First, a
smooth manufactured solution is used to verify the asymptotic behavior
predicted by the analysis for a fixed diffusion coefficient. Second,
the qualitative behavior of the DD formulation is examined for a
strongly advection-dominated problem with unresolved outflow layers.

Throughout this section, we monitor both the resolved component
$u_h=\kappa_h(u_{hb})$ and the enriched approximation $u_{hb}$. The
corresponding errors are denoted by
\begin{align*}
	e_{0,h}
	&:=\|u-u_h\|_0,
	&
	e_{1,h}
	&:=|u-u_h|_1,
	&
	E_h
	&:=\|u-u_h\|_E,
	\\
	e_{0,hb}
	&:=\|u-u_{hb}\|_0,
	&
	e_{1,hb}
	&:=|u-u_{hb}|_1,
	&
	E_{hb}
	&:=\|u-u_{hb}\|_E.
\end{align*}
We also evaluate the artificial dissipation $A_h$ and its associated
square-root measure $S_h$, defined in
\eqref{eq:Ah_definition} and \eqref{eq:Sh_definition}, respectively.
For comparison with the estimate of
Theorem~\ref{thm:combined_hb}, we introduce
\[
Q_{hb}:=E_{hb}+S_h.
\]
The quantities $Q_{hb}$ and
$\left(E_{hb}^2+A_h\right)^{1/2}$ are equivalent, since
\begin{equation*}
	\left(E_{hb}^2+A_h\right)^{1/2}
	\le
	Q_{hb}
	\le
	\sqrt{2}
	\left(E_{hb}^2+A_h\right)^{1/2}.
\end{equation*}
Consequently, they have the same asymptotic order.

For fixed problem data and fixed $\varepsilon>0$, the analysis predicts
\begin{equation*}
	E_h=O(h),
	\qquad
	E_{hb}=O(h),
	\qquad
	A_h=O(h^2),
	\qquad
	S_h=O(h),
	\qquad
	Q_{hb}=O(h).
\end{equation*}
The $L^2$ errors are also reported as numerical diagnostics. Since no
$L^2$ error estimate is established in the present work, the observed
$L^2$ rates are not used as verification of a theoretical
claim.

\subsection{Computational setting}
\label{subsec:computational_setting}

All computations use the same characteristic length and Dynamic Diffusion
scaling,
\begin{equation*}
h_T=\sqrt{|T|},
\qquad
\mu(h_T)=h_T.
\end{equation*}
On the shape-regular triangular meshes considered here,
$\sqrt{|T|}$ is uniformly equivalent to
$\operatorname{diam}(T)$, so this choice is consistent with the
characteristic mesh size employed in the theoretical estimates.

In both experiments, we use the structured mesh pattern denoted by Grid~1
in \cite{Santos-etal:2021}: the square domain is partitioned into congruent
squares, each divided along the diagonal connecting its lower-left and
upper-right vertices. The influence of different mesh patterns on the same
DD formulation was examined in \cite{Santos-etal:2021}, where comparable
convergence orders were observed across four structured and unstructured
grids. We therefore restrict the present experiments to this single mesh
family, since their primary purpose is to corroborate the theoretical results
established in the preceding sections.

The nonlinear DD problem is solved by a direct fixed-point iteration.
For any $v_h\in V_h^0$, let
\begin{equation*}
	\boldsymbol{\xi}(v_h)
	:=
	\left(
	\xi_T(v_h)
	\right)_{T\in\mathcal T_h}
\end{equation*}
denote the vector obtained by evaluating the local artificial-diffusion
coefficient on every element. The coefficient vector used at iteration
$k$ is denoted by
\[
\boldsymbol{\xi}^{(k)}
:=
\left(
\xi_T^{(k)}
\right)_{T\in\mathcal T_h}.
\]
The iteration is initialized with $\boldsymbol{\xi}^{(0)}=\boldsymbol{0}$. For a given $\boldsymbol{\xi}^{(k)}$, the corresponding statically
condensed linear problem is solved to obtain the resolved approximation
$u_h^{(k)}$. The coefficient vector for the next iteration is then
computed directly from this solution:
\begin{equation}
	\boldsymbol{\xi}^{(k+1)}
	=
	\boldsymbol{\xi}\bigl(u_h^{(k)}\bigr).
	\label{eq:fixed_point_update}
\end{equation}
No under-relaxation, coefficient freezing, continuation, or other
auxiliary convergence mechanism is employed. The regularization
parameter is fixed throughout at
\[
\tau=10^{-5}.
\]
The nonlinear iteration is monitored through the relative changes
\begin{align*}
	r_u^{(k)}
	&:=
	\frac{
		\|u_h^{(k+1)}-u_h^{(k)}\|_{\infty}
	}{
		\|u_h^{(k+1)}\|_{\infty}
	},
	&
	r_{\xi}^{(k)}
	&:=
	\frac{
		\|\boldsymbol{\xi}^{(k+1)}
		-\boldsymbol{\xi}^{(k)}\|_2
	}{
		\|\boldsymbol{\xi}^{(k+1)}\|_2
	}.
\end{align*}
Here, $\|\cdot\|_\infty$ and $\|\cdot\|_2$ denote the corresponding
vector norms. In addition, the relative fixed-point defect is defined
by
\begin{equation*}
	r_{\mathrm{FP}}^{(k)}
	:=
	\frac{
		\left\|
		\boldsymbol{\xi}^{(k+1)}
		-
		\boldsymbol{\xi}\bigl(u_h^{(k+1)}\bigr)
		\right\|_2
	}{
		\left\|
		\boldsymbol{\xi}\bigl(u_h^{(k+1)}\bigr)
		\right\|_2
	}.
\end{equation*}
The indicators have complementary interpretations. The quantity
$r_{\xi}^{(k)}$ measures the change in the coefficient vector relative
to the preceding iterate, while $r_{\mathrm{FP}}^{(k)}$ is the
normalized residual of the nonlinear coefficient relation. For the
direct update \eqref{eq:fixed_point_update}, they satisfy
\[
r_{\mathrm{FP}}^{(k)}
=
r_{\xi}^{(k+1)},
\]
so they contain the same coefficient variation with an index shift.
Nevertheless, $r_{\xi}^{(k)}$ is used to monitor the iteration history,
whereas $r_{\mathrm{FP}}^{(k)}$ directly verifies the consistency of
the final solution--coefficient pair. This distinction becomes
essential for relaxed iterations, in which the two indicators are no
longer related by a simple index shift. Although no relaxation is
employed here, both are reported to document the iterative convergence
and the residual of the nonlinear relation separately.

For a mesh sequence
$\{\mathcal T_{h_\ell}\}_{\ell\ge0}$, the observed order associated
with a positive quantity $X_h$ is computed as
\begin{equation*}
	p_X^{(\ell)}
	=
	\frac{
		\log\!\left(X_{h_{\ell-1}}/X_{h_\ell}\right)
	}{
		\log\!\left(h_{\ell-1}/h_\ell\right)
	}.
\end{equation*}

\subsection{Example 1: smooth manufactured solution}
\label{subsec:smooth_problem}

Let $\Omega=(0,1)^2$ and consider
\begin{align*}
-\varepsilon\Delta u
+\boldsymbol{\beta}\cdot\boldsymbol{\nabla}u
+\sigma u
&=f
&&\text{in }\Omega,\\
u&=0
&&\text{on }\partial\Omega,
\end{align*}
with $\boldsymbol{\beta}=(3,2)^T$, $\varepsilon=10^{-2}$ and $\sigma\in\{0,1\}$. The source term is chosen so that
\begin{equation*}
u(x,y)=\sin(\pi x)\sin(\pi y),
\end{equation*}
which gives
$$
f(x,y)
=
\bigl(2\varepsilon\pi^2+\sigma\bigr)
\sin(\pi x)\sin(\pi y)
+
3\pi\cos(\pi x)\sin(\pi y)
+
2\pi\sin(\pi x)\cos(\pi y).
$$
This benchmark is smooth, satisfies the homogeneous boundary condition
exactly, and is independent of $\varepsilon$ and $\sigma$. The diffusion coefficient is held fixed throughout the refinement sequence,
in accordance with the fixed-$\varepsilon$ character of the estimates proved
in the analysis.

We use structured triangular meshes obtained from $N\times N$ Cartesian
subdivisions with
\begin{equation*}
N=12,24,48,96,
\end{equation*}
each square being divided into two triangles. Thus,
$\#\mathcal T_h=2N^2$. Table~\ref{tab:smooth_meshes} reports the resulting
mesh sizes and local P\'eclet numbers. Since $Pe_T>1$ on all four meshes, the nonlinear artificial-diffusion
operator $D_h$ is active throughout the complete refinement sequence.

\begin{table}[htbp]
\centering
\caption{Meshes used in Example~1 for $\varepsilon=10^{-2}$.}
\label{tab:smooth_meshes}
\begin{tabular}{rrrrr}
\toprule
$N$ & Elements & $h$ & $Pe_{\max}$ & DD-active elements \\
\midrule
$12$ & $288$   & $5.8926\times10^{-2}$ & $10.623$ & $100\%$ \\
$24$ & $1152$  & $2.9463\times10^{-2}$ & $5.311$  & $100\%$ \\
$48$ & $4608$  & $1.4731\times10^{-2}$ & $2.656$  & $100\%$ \\
$96$ & $18432$ & $7.3657\times10^{-3}$ & $1.328$  & $100\%$ \\
\bottomrule
\end{tabular}
\end{table}

For this experiment, the nonlinear stopping criterion is
\begin{equation}
r_u^{(k)}<10^{-3},
\qquad
r_{\xi}^{(k)}<10^{-3},
\qquad
r_{\mathrm{FP}}^{(k)}<10^{-4}.
\label{eq:smooth_stopping}
\end{equation}

\subsubsection{Error convergence}

Table~\ref{tab:smooth_L2_errors} reports the $L^2$ errors of the resolved and
enriched approximations. The two approximations display nearly identical
asymptotic behavior, with observed orders approaching two under refinement.
On the finest mesh, the rates are $1.93$ and $1.92$ for $u_h$ and $u_{hb}$
when $\sigma=0$, and $1.99$ and $1.98$ when $\sigma=1$. These values are
reported only as empirical numerical evidence; no $L^2$ convergence estimate
is invoked or inferred from the present analysis.

\begin{table}[htbp]
\centering
\caption{$L^2$ errors for the smooth manufactured problem with
$\varepsilon=10^{-2}$.}
\label{tab:smooth_L2_errors}
\begin{tabular}{crrrrrr}
\toprule
$\sigma$ & Elements
& $e_{0,h}$ & $p_{0,h}$
& $e_{0,hb}$ & $p_{0,hb}$ \\
\midrule
\multirow{4}{*}{$0$}
& $288$ & $4.3757\times10^{-2}$ & -- & $4.3134\times10^{-2}$ & -- \\
& $1152$ & $1.3500\times10^{-2}$ & 1.70 & $1.3334\times10^{-2}$ & 1.69 \\
& $4608$ & $3.7758\times10^{-3}$ & 1.84 & $3.7364\times10^{-3}$ & 1.84 \\
& $18432$ & $9.9253\times10^{-4}$ & 1.93 & $9.8538\times10^{-4}$ & 1.92 \\
\midrule
\multirow{4}{*}{$1$}
& $288$ & $3.9791\times10^{-2}$ & -- & $3.8957\times10^{-2}$ & -- \\
& $1152$ & $1.0416\times10^{-2}$ & 1.93 & $1.0196\times10^{-2}$ & 1.93 \\
& $4608$ & $2.8735\times10^{-3}$ & 1.86 & $2.8249\times10^{-3}$ & 1.85 \\
& $18432$ & $7.2214\times10^{-4}$ & 1.99 & $7.1439\times10^{-4}$ & 1.98 \\
\bottomrule
\end{tabular}%
\end{table}

The energy-norm results are shown in
Table~\ref{tab:smooth_energy_errors}. For $\sigma=0$, the enriched rates are
$0.88$, $1.03$, and $1.25$, while for $\sigma=1$ they are $1.17$, $1.04$,
and $1.44$. The resolved approximation exhibits essentially the same
behavior. The corresponding $H^1$-seminorm orders are
$(0.88,1.02,1.24)$ and $(0.88,1.03,1.25)$ for $u_h$ and $u_{hb}$ when
$\sigma=0$, and $(1.02,0.98,1.41)$ and $(1.02,0.99,1.43)$ when
$\sigma=1$. Thus, both the resolved and enriched approximations are
consistent with the first-order energy behavior predicted by the analysis.
The rates moderately above one on the finest mesh are finite-mesh effects
and are not interpreted as evidence of a higher asymptotic order.

\begin{table}[htbp]
\centering
\caption{Energy errors for Example~1 with $\varepsilon=10^{-2}$.}
\label{tab:smooth_energy_errors}
\begin{tabular}{crrrrrr}
\toprule
$\sigma$ & Elements & $E_h$ & $p_{E_h}$ & $E_{hb}$ & $p_{E_{hb}}$ \\
\midrule
\multirow{4}{*}{$0$}
& $288$ & $7.1307\times10^{-2}$ & -- & $7.2776\times10^{-2}$ & -- \\
& $1152$ & $3.8708\times10^{-2}$ & 0.88 & $3.9487\times10^{-2}$ & 0.88 \\
& $4608$ & $1.9078\times10^{-2}$ & 1.02 & $1.9323\times10^{-2}$ & 1.03 \\
& $18432$ & $8.0913\times10^{-3}$ & 1.24 & $8.1007\times10^{-3}$ & 1.25 \\
\midrule
\multirow{4}{*}{$1$}
& $288$ & $7.5973\times10^{-2}$ & -- & $7.6416\times10^{-2}$ & -- \\
& $1152$ & $3.3536\times10^{-2}$ & 1.18 & $3.3880\times10^{-2}$ & 1.17 \\
& $4608$ & $1.6432\times10^{-2}$ & 1.03 & $1.6509\times10^{-2}$ & 1.04 \\
& $18432$ & $6.1342\times10^{-3}$ & 1.42 & $6.0787\times10^{-3}$ & 1.44 \\
\bottomrule
\end{tabular}
\end{table}

The nonlinear contribution is examined in
Table~\ref{tab:smooth_dissipation_errors}. For $\sigma=0$, the observed
orders of $A_h$ are $2.04$, $2.02$, and $1.86$, while those of $S_h$ are
$1.02$, $1.01$, and $0.93$. For $\sigma=1$, the corresponding values are
$2.23$, $2.05$, $1.83$ and $1.11$, $1.02$, $0.91$. The combined quantity
$Q_{hb}$ exhibits a particularly stable first-order trend: its observed
orders are $1.00$, $1.01$, $0.98$ for $\sigma=0$ and $1.12$, $1.03$,
$0.99$ for $\sigma=1$. Hence, the computations are consistent with
\[
A_h=O(h^2),
\qquad
S_h=O(h),
\qquad
Q_{hb}=O(h).
\]
As a consistency check, $p_S=p_A/2$ up to rounding, as follows directly from
$S_h=A_h^{1/2}$.

Figure~\ref{fig:smooth_DD_convergence_rates} provides a graphical summary
of these convergence trends.

\begin{table}[htbp]
\centering
\caption{Artificial dissipation and combined enriched quantity for
Example~1 with $\varepsilon=10^{-2}$.}
\label{tab:smooth_dissipation_errors}
\begin{tabular}{crrrrrrr}
\toprule
$\sigma$ & Elements
& $A_h$ & $p_A$
& $S_h$ & $p_S$
& $Q_{hb}$ & $p_{Q_{hb}}$ \\
\midrule
\multirow{4}{*}{$0$}
& $288$ & $1.1871\times10^{-1}$ & -- & $3.4455\times10^{-1}$ & -- & $4.1732\times10^{-1}$ & -- \\
& $1152$ & $2.8849\times10^{-2}$ & 2.04 & $1.6985\times10^{-1}$ & 1.02 & $2.0934\times10^{-1}$ & 1.00 \\
& $4608$ & $7.1083\times10^{-3}$ & 2.02 & $8.4311\times10^{-2}$ & 1.01 & $1.0363\times10^{-1}$ & 1.01 \\
& $18432$ & $1.9619\times10^{-3}$ & 1.86 & $4.4294\times10^{-2}$ & 0.93 & $5.2394\times10^{-2}$ & 0.98 \\
\midrule
\multirow{4}{*}{$1$}
& $288$ & $1.1864\times10^{-1}$ & -- & $3.4444\times10^{-1}$ & -- & $4.2086\times10^{-1}$ & -- \\
& $1152$ & $2.5332\times10^{-2}$ & 2.23 & $1.5916\times10^{-1}$ & 1.11 & $1.9304\times10^{-1}$ & 1.12 \\
& $4608$ & $6.1175\times10^{-3}$ & 2.05 & $7.8215\times10^{-2}$ & 1.02 & $9.4723\times10^{-2}$ & 1.03 \\
& $18432$ & $1.7252\times10^{-3}$ & 1.83 & $4.1536\times10^{-2}$ & 0.91 & $4.7615\times10^{-2}$ & 0.99 \\
\bottomrule
\end{tabular}%
\end{table}

\begin{figure}[htbp]
\centering
\begin{minipage}{0.48\textwidth}
\centering
\includegraphics[width=\linewidth]{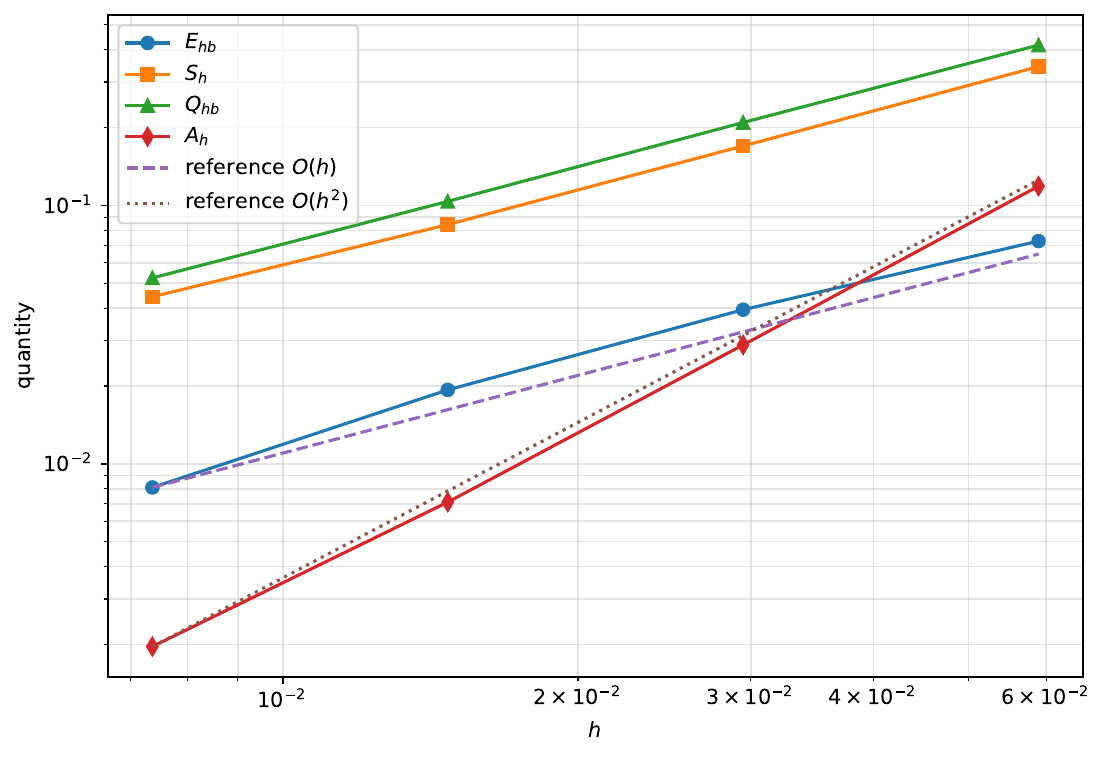}
\end{minipage}\hfill
\begin{minipage}{0.48\textwidth}
\centering
\includegraphics[width=\linewidth]{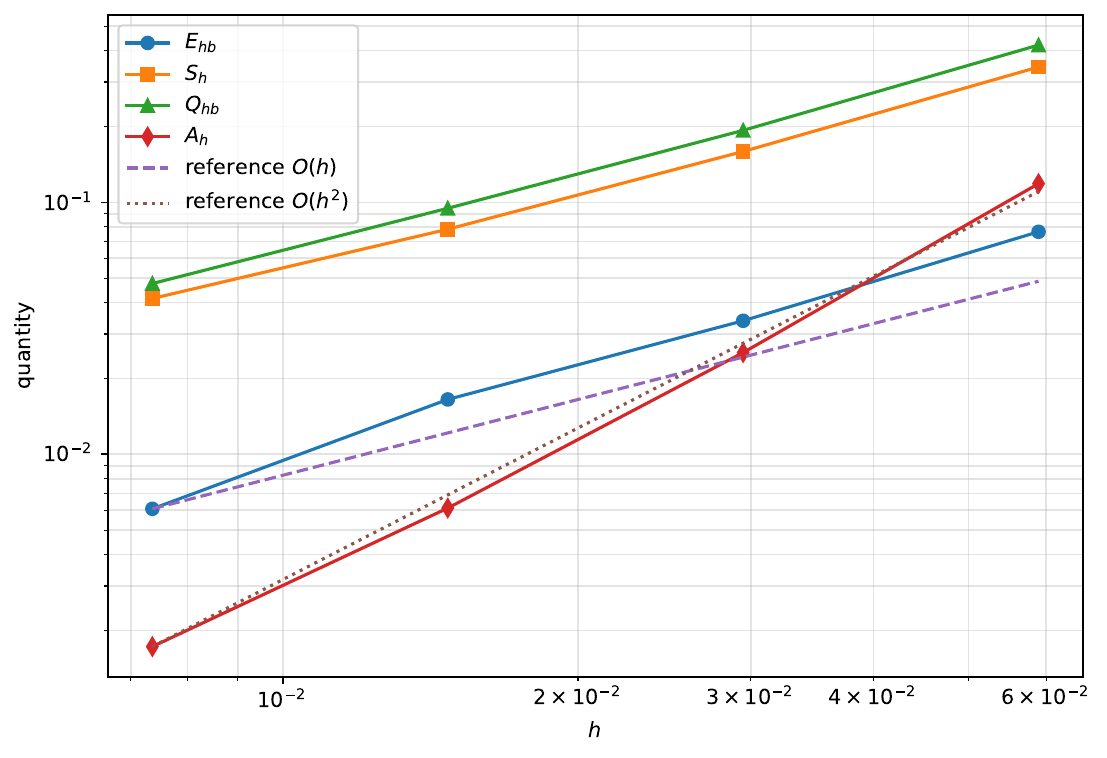}
\end{minipage}
\caption{Convergence of the enriched energy error and the nonlinear DD
quantities for $\varepsilon=10^{-2}$: $\sigma=0$ (left) and $\sigma=1$
(right). Reference slopes $O(h)$ and $O(h^2)$ indicate the rates predicted
by the analysis.}
\label{fig:smooth_DD_convergence_rates}
\end{figure}

Table~\ref{tab:smooth_nonlinear_iterations} reports the nonlinear iteration
counts and the final stopping indicators. All computations satisfy
\eqref{eq:smooth_stopping} using the same fixed-point iteration,
without under-relaxation, freezing, or continuation. The iteration counts vary with the mesh and
with $\sigma$, as expected for a nonlinear coefficient update, but this
variation does not alter the observed spatial convergence trends.

\begin{table}[htbp]
\centering
\caption{Nonlinear iteration diagnostics for Example~1.}
\label{tab:smooth_nonlinear_iterations}
\begin{tabular}{crrrrr}
\toprule
$\sigma$ & Elements & Iter. & $r_u$ & $r_{\xi}$ & $r_{\mathrm{FP}}$ \\
\midrule
\multirow{4}{*}{$0$}
& $288$ & 173 & $5.00e-06$ & $5.14e-05$ & $9.92e-05$ \\
& $1152$ & 316 & $1.56e-05$ & $5.02e-05$ & $9.59e-05$ \\
& $4608$ & 355 & $4.08e-05$ & $5.04e-05$ & $9.64e-05$ \\
& $18432$ & 290 & $3.74e-05$ & $5.82e-05$ & $9.97e-05$ \\
\midrule
\multirow{4}{*}{$1$}
& $288$ & 288 & $1.04e-06$ & $5.05e-05$ & $9.94e-05$ \\
& $1152$ & 145 & $1.33e-05$ & $5.06e-05$ & $9.85e-05$ \\
& $4608$ & 154 & $3.20e-05$ & $5.06e-05$ & $8.72e-05$ \\
& $18432$ & 216 & $1.66e-05$ & $5.36e-05$ & $1.00e-04$ \\
\bottomrule
\end{tabular}
\end{table}

\subsection{Example 2: strongly advection-dominated outflow layers}
\label{subsec:layer_problem}

The second experiment examines the same DD formulation in a regime in which
the physical layers are far thinner than the mesh resolution. We consider
\begin{align*}
-\varepsilon\Delta u
+\boldsymbol{\beta}\cdot\boldsymbol{\nabla}u
&=f
&&\text{in }\Omega=(0,1)^2,\\
u&=0
&&\text{on }\partial\Omega,
\end{align*}
with $\boldsymbol{\beta}=(1,1)^T$ and $\sigma=0$.
Define
\begin{equation*}
g_\varepsilon(t)
=
t-
\frac{e^{t/\varepsilon}-1}
{e^{1/\varepsilon}-1}.
\end{equation*}
The exact solution is
\begin{equation*}
u(x,y)
=
g_\varepsilon(x)\,g_\varepsilon(y).
\end{equation*}
This boundary-layer benchmark has also been considered in recent
analyses of DD formulations~\cite{Du-etal:2026,Du-etal:2026b}. Here it
is used for the DD formulation with $\sigma=0$ and on
a severely under-resolved uniform mesh, in order to illustrate its
behavior in the strongly advection-dominated regime.

Since $g_\varepsilon(0)=g_\varepsilon(1)=0$, the boundary condition is
satisfied on the whole boundary. Moreover,
\[
-\varepsilon g_\varepsilon''(t)
+
g_\varepsilon'(t)
=
1,
\]
and therefore the corresponding source term is
\begin{equation*}
f(x,y)
=
g_\varepsilon(x)+g_\varepsilon(y).
\end{equation*}
For $\varepsilon\ll1$, the solution is close to $xy$ in the interior and
develops exponential outflow layers adjacent to $x=1$ and $y=1$.

We take
\begin{equation*}
\varepsilon=10^{-5}
\end{equation*}
and use a structured $40\times40$ mesh, corresponding to $3200$ triangular
elements. Since
\[
h_T=\frac{1}{40\sqrt{2}},
\qquad
\|\boldsymbol{\beta}\|_{0,\infty,T}=\sqrt{2},
\]
the local P\'eclet number is
\begin{equation*}
Pe_T=1250
\qquad
\forall T\in\mathcal T_h.
\end{equation*}
Thus, the nonlinear artificial-diffusion operator is active throughout the mesh and the computation lies
deep in the advection-dominated regime. 

For this qualitative test, the nonlinear iteration is stopped when
\begin{equation*}
r_u^{(k)}<10^{-2},
\qquad
r_{\xi}^{(k)}<10^{-2},
\qquad
r_{\mathrm{FP}}^{(k)}<10^{-2}.
\end{equation*}
Table~\ref{tab:layer_diagnostics} summarizes the computation.

\begin{table}[htbp]
\centering
\caption{Diagnostic quantities for Example~2.}
\label{tab:layer_diagnostics}
\begin{tabular}{lr}
\toprule
Quantity & Value \\
\midrule
$\varepsilon$ & $10^{-5}$ \\
Elements & $3200$ \\
$Pe_T$ & $1250$ \\
DD-active elements & $100\%$ \\
Nonlinear iterations & $22$ \\
Final $r_u$ & $3.54e-04$ \\
Final $r_{\xi}$ & $5.10e-03$ \\
Final $r_{\mathrm{FP}}$ & $4.80e-03$ \\
$\min u_h$ & $0.0000e+00$ \\
$\max u_h$ & $8.2177e-01$ \\
$\|u-u_h\|_0$ & $8.4439e-02$ \\
\bottomrule
\end{tabular}
\end{table}

Figure~\ref{fig:layer_3d} compares the exact and resolved DD solutions using
the same three-dimensional viewpoint.
\begin{figure}[htbp]
\centering
\begin{minipage}{0.48\textwidth}
\centering
\includegraphics[width=\linewidth]{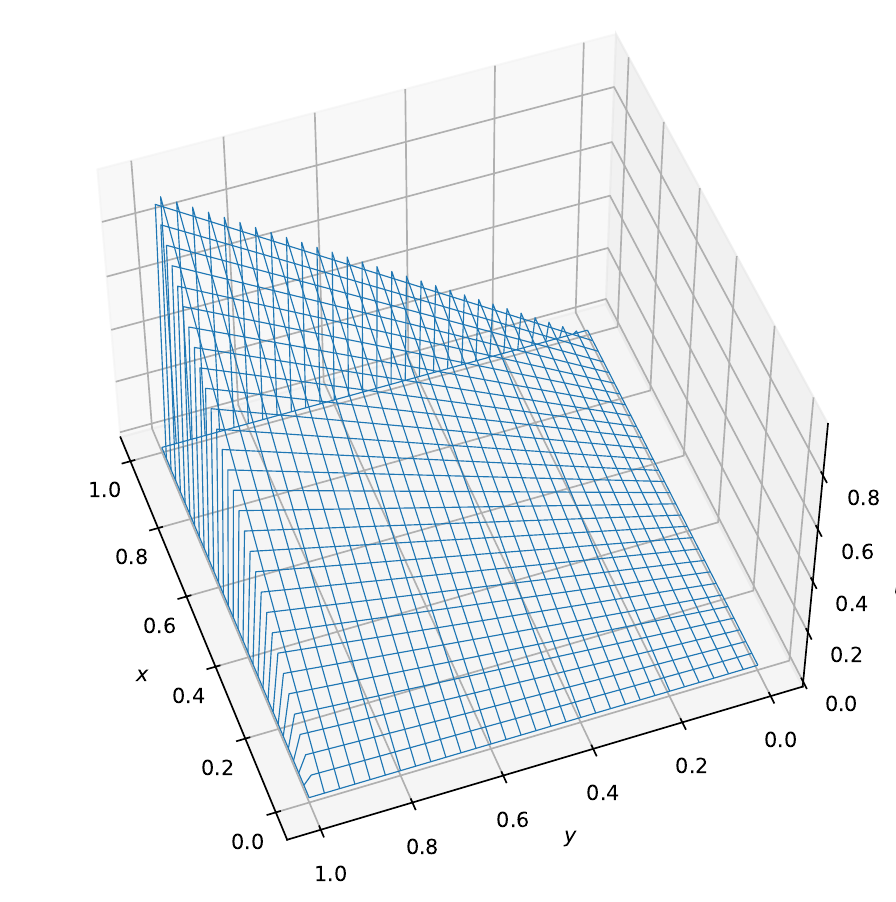}

\small (a) Exact solution
\end{minipage}
\hfill
\begin{minipage}{0.48\textwidth}
\centering
\includegraphics[width=\linewidth]{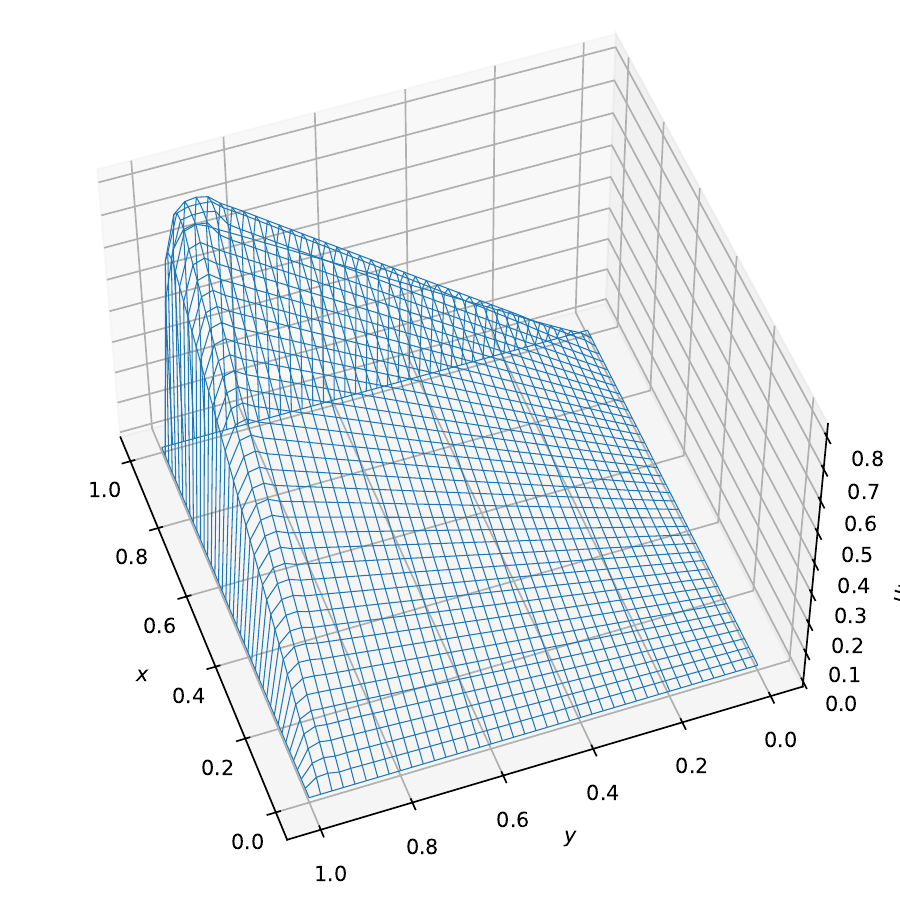}

\small (b) DD approximation
\end{minipage}
\caption{Example~2 with $\varepsilon=10^{-5}$: exact solution and resolved
DD approximation on the $40\times40$ mesh.}
\label{fig:layer_3d}
\end{figure}
The horizontal cut at $y=0.5$ provides a more direct assessment of the
outflow layer. As shown in Figure~\ref{fig:layer_cut}, the DD approximation
tracks the exact solution closely over most of the domain and reproduces the
location of the sharp transition at $x=1$. The transition is numerically
broadened over the last few mesh intervals, as expected because the physical
layer thickness, $O(\varepsilon)$, is several orders of magnitude smaller
than the element size.

\begin{figure}[htbp]
\centering
\includegraphics[width=0.72\textwidth]{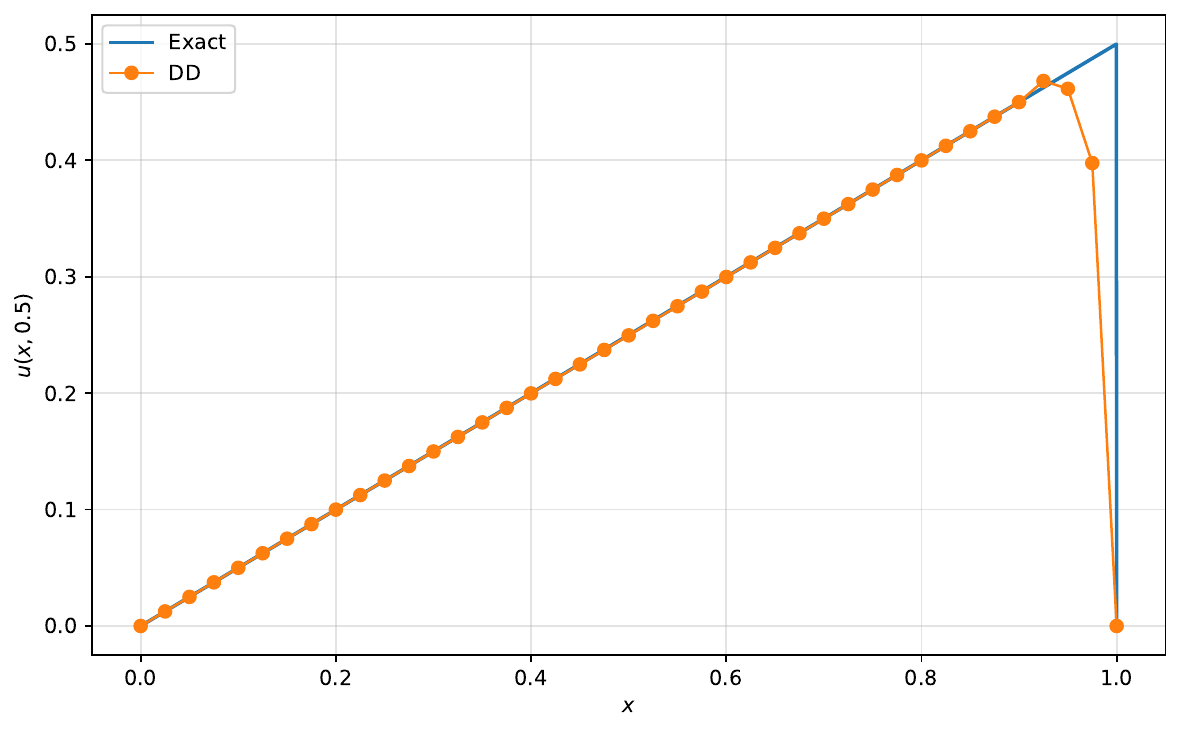}
\caption{Example~2: horizontal cut at $y=0.5$ comparing the exact solution
and the resolved DD approximation for $\varepsilon=10^{-5}$.}
\label{fig:layer_cut}
\end{figure}

No undershoot, overshoot, or visible spurious oscillation is observed in this
computation, and the approximation remains smooth in the interior. The price
paid for the added nonlinear diffusion is a noticeable broadening of the
unresolved outflow layer.
The experiment shows that the same DD coefficient scaling used in Example~1
yields a stable approximation in a severely under-resolved
advection-dominated regime, with no additional tuning and with convergence of
the fixed-point iteration according to all three stopping indicators.

Together, Examples~1 and~2 address two distinct numerical questions. The first
tests the fixed-$\varepsilon$ asymptotic behavior predicted by the analysis,
whereas the second examines the same nonlinear artificial-diffusion mechanism
in a severely under-resolved transport regime. In both cases, the coefficient
scaling, regularization parameter, and nonlinear
iteration are kept unchanged.

\paragraph{Code availability.}
The Python code used to generate the numerical results reported in this
section is publicly available at

\begin{center}
	\url{https://github.com/isaacpsantos/DD-Convergence-Analysis-2026}
\end{center}

\noindent
The repository contains a self-contained notebook that can be executed
directly in Google Colab.

\section{Conclusions}

This work advances the mathematical analysis of the Dynamic Diffusion
(DD) finite element formulation studied in
\cite{Santos-etal:2021}, without modifying its discrete structure. We
first establish an explicit local Lipschitz estimate for the artificial
diffusivity on the resolved-scale space, with a constant of order
$h_T$. The stability properties of the resolved-scale projection then
extend this estimate to the full enriched finite element space. The
resulting bound yields uniqueness of the discrete solution for
sufficiently fine meshes.

For continuous piecewise linear finite elements enriched with simplex
bubble functions, the enriched DD approximation and its resolved
component converge with the optimal first-order rate in the energy
norm. Moreover, the square-root artificial-dissipation measure
$S_h=A_h^{1/2}$ satisfies $S_h=O(h)$ while active elements remain
present. Once the local Péclet criterion deactivates the artificial
diffusion throughout the mesh, $A_h$ vanishes exactly. Consequently,
the energy-error and artificial-dissipation measure is also $O(h)$,
sharpening the previously available $O(h^{1/2})$ estimate for the DD
formulation.

The numerical experiments corroborate the analytical predictions for
both $\sigma=0$ and $\sigma=1$. For the smooth manufactured problem,
the computed solutions exhibit the expected first-order behavior in
the energy norm, while the square-root artificial-dissipation measure
displays the decay predicted by the analysis. The local Péclet number
remains larger than one throughout the mesh sequence, so the observed
convergence occurs while the nonlinear artificial-diffusion operator
$D_h$ remains active. The computations employ the DD formulation
without under-relaxation, coefficient freezing, continuation, or any
other auxiliary mechanism for the nonlinear iteration. The $L^2$
errors also exhibit second-order behavior, although no $L^2$ error
estimate is established in the present work.

An advection-dominated problem with sharp outflow layers
further illustrates the stabilizing behavior of the DD formulation in
a regime characterized by a large local Péclet number. Although this
experiment is not intended as an additional convergence test, it
complements the theoretical analysis by showing that the method remains
effective in the presence of strongly advective features.

These results close the gap between the previously available a priori
bound and the first-order behavior observed for the DD formulation.
They also show that optimal finite element accuracy is compatible with
the nonlinear, locally activated artificial-diffusion mechanism.
Further developments may address locally refined uniqueness conditions
and parameter-robust estimates for singularly perturbed problems.

\section*{Declaration of Generative AI and AI-Assisted Technologies in the Writing Process}

During the preparation of this work, the author used ChatGPT (OpenAI) for language polishing and editorial organization of the manuscript. After using this tool, the author thoroughly reviewed and edited the content as needed and takes full and exclusive responsibility for the scientific content of the published article.

\bibliographystyle{unsrt}
\bibliography{references}

\end{document}